\documentclass[10pt]{amsart}
\usepackage{graphicx}
\usepackage{amscd}
\usepackage{amsmath}
\usepackage{amsthm}
\usepackage{amsfonts}
\usepackage{amssymb}
\usepackage{mathrsfs}
\usepackage{bm}
\usepackage{enumerate}
\usepackage{amsrefs}
\usepackage{xcolor}
\usepackage[colorlinks, citecolor=blue, linkcolor=red, pdfstartview=FitB]{hyperref}
\usepackage[latin1]{inputenc}

\usepackage{marginnote}	
\usepackage{soul} 			

\newcommand{\bC}{{\mathbb{C}}}
\newcommand{\bD}{{\mathbb{D}}}

\newcommand{\bM}{{\mathbb{M}}}
\newcommand{\bN}{{\mathbb{N}}}

\newcommand{\bT}{{\mathbb{T}}}

  \newcommand{\A}{{\mathcal{A}}}
  \newcommand{\B}{{\mathcal{B}}}
  \newcommand{\C}{{\mathcal{C}}}
  \newcommand{\D}{{\mathcal{D}}}
  \newcommand{\E}{{\mathcal{E}}}
  \newcommand{\F}{{\mathcal{F}}}
  \newcommand{\G}{{\mathcal{G}}}
\renewcommand{\H}{{\mathcal{H}}}
  
  \newcommand{\J}{{\mathcal{J}}}
  \newcommand{\K}{{\mathcal{K}}}  

  \newcommand{\M}{{\mathcal{M}}}
  
\renewcommand{\O}{{\mathcal{O}}}

  \newcommand{\R}{{\mathcal{R}}}
\renewcommand{\S}{{\mathcal{S}}}
  \newcommand{\T}{{\mathcal{T}}}
  \newcommand{\U}{{\mathcal{U}}}

\newcommand{\fA}{{\mathfrak{A}}}

\newcommand{\fB}{{\mathfrak{B}}}

\newcommand{\fF}{{\mathfrak{F}}}

\newcommand{\fJ}{{\mathfrak{J}}}
\newcommand{\fK}{{\mathfrak{K}}}

\newcommand{\fL}{{\mathfrak{L}}}
\newcommand{\fM}{{\mathfrak{M}}}

\newcommand{\fQ}{{\mathfrak{Q}}}
\newcommand{\fR}{{\mathfrak{R}}}

\newcommand{\fS}{{\mathfrak{S}}}

\newcommand{\fT}{{\mathfrak{T}}}

\newcommand{\rA}{\mathrm{A}}

\newcommand{\rC}{\mathrm{C}}

\renewcommand{\phi}{\varphi}

\newcommand{\upchi}{{\raise.35ex\hbox{$\chi$}}}

\newcommand{\ol}{\overline}

\newcommand{\sot}{\textsc{sot}}

\newtheorem{lemma}{Lemma}[section]
\newtheorem{theorem}[lemma]{Theorem}
\newtheorem{proposition}[lemma]{Proposition}
\newtheorem{corollary}[lemma]{Corollary}

\theoremstyle{definition}

\newtheorem{remark}[lemma]{Remark}
\newtheorem{question}{Question}
\newtheorem{example}{Example}

\date{\today}
\author{Ian Thompson}

\address{Department of Mathematics, University of Manitoba, Winnipeg, Manitoba, Canada R3T 2N2}

\email{thompsoi@myumanitoba.ca\vspace{-2ex}}
\thanks{The author was partially supported by an NSERC CGS-M Scholarship.}

\title[Maximal RFD $\rC^*$-covers]{Maximal $\rC^*$-covers and residual finite-dimensionality}
\begin{document}
\begin{abstract}
We study residually finite-dimensional (or RFD) operator algebras which may not be self-adjoint. An operator algebra may be RFD while simultaneously possessing completely isometric representations whose generating $\rC^*$-algebra is not RFD. This has provided many hurdles in characterizing residual finite-dimensionality for operator algebras. To better understand the elusive behaviour, we explore the $\rC^*$-covers of an operator algebra. First, we equate the collection of $\rC^*$-covers with a complete lattice arising from the spectrum of the maximal $\rC^*$-cover. This allows us to identify a largest RFD $\rC^*$-cover whenever the underlying operator algebra is RFD. The largest RFD $\rC^*$-cover is shown to be similar to the maximal $\rC^*$-cover in several different facets and this provides supporting evidence to a previous query of whether an RFD operator algebra always possesses an RFD maximal $\rC^*$-cover. In closing, we present a non self-adjoint version of Hadwin's characterization of separable RFD $\rC^*$-algebras.
\end{abstract}
\maketitle

\section{Introduction} \label{S:Introduction}

The characterization of a $\rC^*$-algebra can often be reduced to studying its representation theory. To this end, a natural choice is to investigate the $\rC^*$-algebras which may be completely understood by their finite-dimensional $*$-representations. This philosophy underlies nuclearity and quasidiagonality of $\rC^*$-algebras, which have generated a flurry of recent progress (see \cite{gong2014classification}, \cite{tikuisis2017quasidiagonality}). However, where those notions constitute an approximation property which is `internal' to the $\rC^*$-algebra, we study a finite-dimensional approximation property external to the $\rC^*$-algebra. 

Namely, we study the notion of residual finite-dimensionality. Residually finite-dimensional $\rC^*$-algebras constitute a truly significant class of the quasidiagonal $\rC^*$-algebras and possess numerous equivalent characterizations (see \cite{archbold1995residually}, \cite{clouatre2020finite}, \cite{courtney2019elements}, \cite{exel1992finite}, \cite{hadwin2014lifting}). But instead, we will consider residual finite-dimensionality of (possibly non self-adjoint) operator algebras. We say that an operator algebra is \emph{residually finite-dimensional} (or RFD) if it may be embedded into a direct product of finite-dimensional $\rC^*$-algebras.  This concept was recently introduced in \cite{clouatre2019residually} and there are still only a few characterizations of RFD operator algebras (see \cite{clouatre2021finite}, \cite{clouatre2019residually}).

Operator algebras possess quite idiosyncratic behaviour up to complete isometric isomorphism and this poses many difficulties in characterizing residual finite-dimensionality. We will clarify this hurdle. Suppose $\A\subset B(\H)$ is an operator algebra. Upon passing through a completely isometric representation $\iota:\A\rightarrow B(\K)$, one may find an isomorphic copy of the same operator algebra. However, there may be a stark difference between the $\rC^*$-algebras generated by $\A$ and $\iota(\A)$, which we denote by $\rC^*(\A)$ and $\rC^*(\iota(\A))$, respectively. It is not necessarily true that $\rC^*(\A)$ and $\rC^*(\iota(\A))$ are $*$-isomorphic. It can even be the case that $\rC^*(\A)$ is RFD whereas $\rC^*(\iota(\A))$ possesses no finite-dimensional $*$-representations (Example \ref{E:cuntz}). In this sense, it is not clear how to detect residual finite-dimensionality of operator algebras.

We consider the collection of all $\rC^*$-covers of an operator algebra $\A$ to study residual finite-dimensionality. A \emph{$\rC^*$-cover} of $\A$ is a pair $(\fA, \iota)$ where $\iota$ is a completely isometric representation of $\A$ and $\fA = \rC^*(\iota(\A))$. In the literature on $\rC^*$-covers, there are two well-studied covers which dominate the theory: the minimal $\rC^*$-cover and the maximal $\rC^*$-cover. These two $\rC^*$-covers always exist for any operator algebra. Existence of the maximal $\rC^*$-cover is shown in \cite[Proposition 2.4.2]{blecher2004operator} and will be quite relevant for our purposes. On the other hand, existence of the minimal $\rC^*$-cover (or $\rC^*$-envelope) has been the basis of a research project spanning nearly half a century (see \cite{arveson2008noncommutative}, \cite{arveson1969subalgebras}, \cite{davidson2015choquet}, \cite{dritschel2005boundary}, \cite{hamana1979injective}, \cite{muhly1998algebraic}). The minimal $\rC^*$-cover will make appearances but, for our purposes, we focus heavily on the maximal $\rC^*$-cover.

The maximal $\rC^*$-cover of an operator algebra $\A$, denoted $(\rC^*_{max}(\A), \mu)$, satisfies the universal property that every representation of $\A$ lifts to a unique $*$-representation of $\rC^*_{max}(\A)$. Hence, the representations of $\A$ are in one-to-one correspondence with the $*$-representations of $\rC^*_{max}(\A)$. As residual finite-dimensionality is encoded in the structural properties of spaces of representations, one could ask whether residual finite-dimensionality of operator algebras is closely linked to that of their maximal $\rC^*$-covers. This provides us with our main inquiry:

\begin{question}\label{Q:main}
If $\A$ is an RFD operator algebra, then is $\rC^*_{max}(\A)$ necessarily RFD?
\end{question}

There have been several instances for which Question \ref{Q:main} has been answered in the affirmative \cite{clouatre2021finite}, \cite{clouatre2019residually}. Most notably, it was shown that $\rC^*_{max}(\A)$ is RFD whenever $\A$ is finite-dimensional. The maximal $\rC^*$-cover is also RFD for Popescu's disk algebra, the Schur-Agler class of functions, as well as some semigroup algebras and spaces of analytic functions. We remark that Question \ref{Q:main} is also related to a result of Pestov \cite[Theorem 4.1]{pestov1994operator}.

Now we outline the structure of the paper and our approach to the main question. Within Section \ref{S:prelim}, we introduce background material on the spectrum of a $\rC^*$-algebra, $\rC^*$-covers and residual finite-dimensionality.

In Section \ref{S:TopOrder}, we study a natural partial ordering on the collection of all $\rC^*$-covers of a fixed operator algebra. This was previously observed in \cite{hamidi2019admissibility}. Under a natural interpretation, we show that the partial ordering is equivalent to the partial ordering given by inclusion of the spectra of the $\rC^*$-covers (Proposition \ref{P:toporder}). Moreover, this produces an isomorphism of complete lattices (Theorem \ref{T:LatticeIsom}) and allows us to equate topological statements on the spectrum of the maximal $\rC^*$-cover with order-theoretic statements on the collection of $\rC^*$-covers. This identification supplies us with a natural framework to study residual finite-dimensionality in Section \ref{S:RFDmax}.

Fix an RFD operator algebra $\A$. In Subsection \ref{SS:abstract}, we derive the existence of a largest RFD $\rC^*$-cover of $\A$ relative to the partial ordering of $\rC^*$-covers (Theorem \ref{T:existRFDmax}). This is referred to as the \emph{RFD-maximal $\rC^*$-cover}, denoted $(\fR(\A), \upsilon)$. As a consequence, Question \ref{Q:main} has an affirmative answer precisely when the RFD-maximal $\rC^*$-cover is the maximal $\rC^*$-cover (Corollary \ref{C:RFDmaxiffmaxRFD}). In Theorem \ref{T:topRFDmax}, we identify the spectrum of the RFD-maximal $\rC^*$-cover relative to the maximal $C^*$-cover.

In Subsection \ref{SS:concrete}, we provide an explicit representation of the RFD-maximal $\rC^*$-cover. Typically, the embedding $\mu:\A\rightarrow\rC^*_{max}(\A)$ is taken to be an appropriately large direct sum of completely contractive representations of $\A$. We show that the RFD-maximal $\rC^*$-cover can be constructed in an analogous way and satisfies a universal property (Theorem \ref{T:concreterepn}):

\begin{theorem}\label{T:intro1}
Let $\A$ be an RFD operator algebra. For any representation $\rho:\A\rightarrow B(\H_\rho)$ on a finite-dimensional Hilbert space, there is a unique $*$-representation $\theta:\fR(\A)\rightarrow B(\H_\rho)$ such that $\theta\circ\upsilon = \rho$. Moreover, $(\fR(\A), \upsilon)$ is minimal among all $\rC^*$-covers of $\A$ which satisfy this property.
\end{theorem}

As previously noted, to answer Question \ref{Q:main} in the affirmative we must show that the RFD-maximal $\rC^*$-cover and the maximal $\rC^*$-cover coincide. In Subsection  \ref{SS:RFDmaxAlgProp}, we exhibit several aspects of the RFD-maximal $\rC^*$-cover which are reminiscent of known properties of the maximal $\rC^*$-cover. Thus, this is consistent with Question \ref{Q:main} having an affirmative answer. For unital operator algebras, the maximal $\rC^*$-cover is known to respect countable direct sum and the free product of finitely many operator algebras \cite[Proposition 2.2]{blecher1999modules}, \cite[Theorem 5.2]{clouatre2019residually}. We show that the RFD-maximal $\rC^*$-cover also preserves these constructions (Theorems \ref{T:RFDmaxDirectSum} and \ref{T:freeproduct}). In turn, this provides supporting evidence for an affirmative answer to Question \ref{Q:main}.

\begin{theorem}\label{T:intro2}
The following statements hold:\begin{enumerate}[{\rm (i)}]
\item If $\A_n, n\in\bN,$ are unital RFD operator algebras, $\fR(\bigoplus_{n=1}^\infty \A_n)\cong \bigoplus_{n=1}^\infty \fR(\A_n)$.
\item If $\A, \B$ are unital RFD operator algebras, then $\fR(\A*\B)\cong \fR(\A)*\fR(\B).$
\end{enumerate}
\end{theorem}

In Section \ref{S:RFDrepns}, we study the representation theory of RFD operator algebras and relate this to properties of the RFD-maximal $\rC^*$-cover. Due to work of Exel and Loring, a $\rC^*$-algebra is RFD precisely when all $*$-representations are point-strong limits of (possibly degenerate) finite-dimensional representations \cite[Theorem 2.4]{exel1992finite}. In this setting, the point-strong and point-strong$*$ topologies coincide. This subtlety regarding the adjoint is responsible for some of the phenomena we witness.

Specifically, the possible discrepancy of the point-strong and point-strong$*$ topologies for operator algebras results in some uncertainty over how to interpret this condition. Here we consider two possible candidates: those representations which are point-strong limits of finite-dimensional representations are referred to as \emph{residually finite-dimensional} representations (or RFD representations). Alongside RFD representations, we consider representations which are point-strong$*$ limits of finite-dimensional representations, called \emph{$*$-residually finite-dimensional} representations (or $*$-RFD representations).

Recently, Clou\^{a}tre and Dor-On showed that $\rC^*_{max}(\A)$ is RFD precisely when every representation of $\A$ is $*$-RFD \cite[Theorem 3.3]{clouatre2021finite}. Utilizing this result, we are presented with a natural intermediate question. Indeed, for Question \ref{Q:main} to have an affirmative answer, it must be the case that RFD and $*$-RFD representations of an RFD operator algebra coincide. This provides us with a suitable motivation for considering both classes of representations.

We wish to study which $\rC^*$-covers RFD and $*$-RFD representations lift to. Similar to the concrete construction of the RFD-maximal $\rC^*$-cover, we may build two $\rC^*$-covers concretely from appropriately large direct sums of RFD representations and $*$-RFD representations. We denote the associated $\rC^*$-covers by $(\fR_s(\A), \upsilon_s)$ and $(\fR_{*s}(\A), \upsilon_{*s})$, respectively. In Corollary \ref{C:RFDmaxLiftsStarRFD}, we show that $(\fR_{*s}(\A), \upsilon_{*s})$ is the RFD-maximal $\rC^*$-cover. Consequently, the $*$-RFD representations of $\A$ are precisely those representations which lift to the RFD-maximal $\rC^*$-cover. We then obtain equivalent conditions for RFD and $*$-RFD representations of an RFD operator algebra to coincide (Theorem \ref{T:starRFDequivalence}). For instance, this is equivalent to $(\fR_s(\A), \upsilon_s)$ being the RFD-maximal $\rC^*$-cover.
 
In Section \ref{S:Hadwin}, we study Hadwin's characterization of separable RFD $\rC^*$-algebras \cite{hadwin2014lifting} and provide a non self-adjoint version of their result. We recount the details.

Let $\{e_n: n\in\bN\}$ be an orthonormal basis for $\ell^2$. For each $n\in\bN$, let $P_n$ be the orthogonal projection onto the linear span of $\{ e_1, \ldots, e_n\}$ and let $\M_n = P_n B(\ell^2)P_n$. Further, let $\fB$ be the $\rC^*$-subalgebra of $\prod_{n=1}^\infty\M_n$ consisting of all sequences $(T_n)_{n\geq1}$ which converge $*$-strongly in $B(\ell^2)$. Let $\pi:\fB\rightarrow B(\ell^2)$ denote the $*$-representation defined by $\pi((T_n)) = *\sot\lim_n T_n.$ A representation $\rho:\A\rightarrow B(\ell^2)$ is \emph{$*$-liftable in the sense of Hadwin} if there is a representation $\tau:\A\rightarrow\fB$ such that $\pi\circ\tau = \rho.$

Hadwin showed that a separable $\rC^*$-algebra $\fA$ is RFD if and only if every unital $*$-representation $\sigma: \widetilde{\fA}\rightarrow B(\ell^2)$ is $*$-liftable \cite[Theorem 11]{hadwin2014lifting}. We show the following in Theorem \ref{T:nonSAHadwin}:

\begin{theorem}\label{T:intro3}
Let $\A$ be a separable operator algebra. Then, $\rC^*_{max}(\A)$ is an RFD $\rC^*$-algebra if and only if every unital representation $\rho:\widetilde{\A}\rightarrow B(\ell^2)$ is $*$-liftable in the sense of Hadwin.
\end{theorem}

\textbf{Acknowledgements.} The author would like to thank their advisor, Rapha\"{e}l Clou\^{a}tre, who was tremendously helpful during the preparation of this manuscript.

\section{Preliminaries} \label{S:prelim}

An \emph{operator algebra} $\A$ is a subalgebra of $B(\H)$ where $\H$ is some Hilbert space. The collection of $n\times n$ complex-valued matrices will be denoted by $\bM_n$ while the collection of $n\times n$ matrices with entries in $\A$ is denoted by $\bM_n(\A)$. For a linear map $\rho: \A\rightarrow B(\H_\rho)$, we let \[\rho^{(n)}: \bM_n(\A)\rightarrow B(\H_\rho^{(n)}), \ \ \ \ \ \ [a_{ij}]\mapsto[\rho(a_{ij})].\]A linear map $\rho$ is \emph{completely contractive} whenever $\rho^{(n)}$ is contractive for each $n\in\bN$. We also refer to a map $\rho$ as being \emph{completely isometric} if $\rho^{(n)}$ is isometric for each $n\in\bN$. A \emph{representation} of an operator algebra is a completely contractive algebra homomorphism $\rho:\A\rightarrow B(\H_\rho)$. When $\A$ is self-adjoint, $\A$ is a $\rC^*$-algebra and a representation is referred to as a \emph{$*$-representation}.

If $\A$ is a non-unital operator algebra, we let $\widetilde{\A}$ denote the unitization. If $\A$ is unital, we will simply refer to $\widetilde{\A}$ as $\A$ itself. We recall the relevant details from \cite{meyer2001adjoining} on unitizations of operator algebras: given a non-unital operator algebra $\A$, there is a completely isometric representation $\A\rightarrow\widetilde{\A}$. Moreover, any representation $\rho:\A\rightarrow B(\H_\rho)$ extends uniquely to a unital representation $\rho^+:\widetilde{\A}\rightarrow B(\H_\rho)$ and $\rho$ is completely isometric precisely when $\rho^+$ is completely isometric.

\subsection{Spectral Topology} \label{SS:topology}

Throughout our analyses, we work with topologies on spaces of representations. These representations will be defined on either $\rC^*$-algebras or non self-adjoint operator algebras. Here, we recount basic facts on the spectrum of a $\rC^*$-algebra (see \cite[Chapter 3]{diximier1977c} for details).

Let $\fA$ be a $\rC^*$-algebra. An ideal $\fJ$ is \emph{primitive} whenever $\fJ$ is the kernel of an irreducible $*$-representation of $\fA$. We define the \emph{primitive ideal space}, denoted Prim$(\fA)$, to be the collection of all primitive ideals. The primitive ideal space is equipped with a natural topology: whenever $\J$ is a collection of primitive ideals, the closure of $\J$ is the collection of all primitive ideals which contain $\bigcap_{\fJ\in\J}\fJ$. The \emph{spectrum} of $\fA$, denoted $\widehat{\fA}$, is the collection of unitary equivalence classes of irreducible $*$-representations of $\fA$. For an irreducible $*$-representation $\pi:\fA\rightarrow B(\H)$, we let $[\pi]$ denote the unitary equivalence class of $\pi$. The topology on $\widehat{\fA}$ is defined to be the weakest topology such that the natural mapping \[\widehat{\fA}\rightarrow\text{Prim}(\fA), \ \ \ \ [\pi]\mapsto\ker\pi,\]is continuous. If $\pi$ and $\sigma$ are $*$-representations of $\fA$, then $\pi$ is \emph{weakly contained} in $\sigma$, denoted $\pi\prec\sigma$, if $\ker\sigma\subset\ker\pi$. Equivalently, $\pi\prec\sigma$ if and only if there is a $*$-homomorphism $\Lambda:\sigma(\fA)\rightarrow\pi(\fA)$ such that $\Lambda\circ\sigma=\pi$. Further, this is equivalent to stipulating that $\lVert\pi(t)\rVert\leq\lVert\sigma(t)\rVert$ for each $t\in\fA$. When $\pi$ and $\sigma$ are weakly contained in one another, we say the representations are \emph{weakly equivalent}. The closure of a singleton $[\pi]\in\widehat{\fA}$ is $\{[\sigma]\in\widehat{\fA} : \sigma\prec\pi\}.$ Let $\D\subset\widehat{\fA}$ and $\D_0$ be a set of irreducible $*$-representations such that $\D = \{ [\sigma]\in\widehat{\fA} : \sigma\in\D_0\}$. Note that if the $*$-representation $\bigoplus_{\sigma\in\D_0}\sigma$ is injective, then the subset $\D$ is dense within $\widehat{\fA}$.

Every irreducible $*$-representation of a closed two-sided ideal $\fJ\subset\fA$ extends uniquely to $\fA$. Conversely, if $\sigma$ is an irreducible $*$-representation of $\fA$, then $\sigma\mid_\fJ$ is irreducible if and only if $\sigma\mid_\fJ\neq 0$ \cite[Lemmata 1.9.14-15]{davidson1996c}. These facts allow for a description of the spectrum which will be central to our arguments:

\begin{theorem}{\cite[Propositions 3.2.1-2]{diximier1977c}}\label{T:DiximierTopology}
Let $\fJ$ be a closed two-sided ideal of a $\rC^*$-algebra $\fA$. Let $\G_\fJ = \{ [\sigma]\in\widehat{\fA}: \sigma\mid_\fJ = 0\}$ and $\U_\fJ = \{[\sigma]\in\widehat{\fA}: \sigma\mid_\fJ\neq0\}$.  Then, the following statements hold:\begin{enumerate}[{\rm (i)}]
\item $\widehat{\fA}$ is the disjoint union of $\U_\fJ$ and $\G_\fJ$;
\item $\U_\fJ$ is an open subset and $\G_\fJ$ is a closed subset in $\widehat{\fA}$;
\item the natural mapping $\G_\fJ\rightarrow\widehat{\fA/\fJ}$ is a homeomorphism;
\item the natural mapping $\U_\fJ\rightarrow\widehat{\fJ}$ is a homeomorphism.
\end{enumerate}Moreover, the mapping $\fJ\mapsto\U_\fJ$ is an inclusion-preserving bijection between the closed two-sided ideals of $\fA$ and the open subsets of $\widehat{\fA}$.
\end{theorem}

\subsection{$\rC^*$-covers} \label{SS:CstarCovers}

The $\rC^*$-algebra that an operator algebra $\A\subset B(\H)$ generates often carries strikingly different information based off the choice of representation for the algebra. We will make this explicit. Within $B(\H)$, there is a smallest $\rC^*$-algebra containing $\A$, denoted $\rC^*(\A)$. Given a completely isometric representation $\iota:\A\rightarrow B(\K)$, the $\rC^*$-algebra $\rC^*(\iota(\A))$ could be quite unlike $\rC^*(\A)$ (Example \ref{E:DiscAlg}).

Let $\A\subset B(\H)$ be an operator algebra. We call the pair $(\fA, \iota)$ a \emph{$\rC^*$-cover} of $\A$ if $\iota:\A\rightarrow B(\K)$ is a completely isometric representation and $\fA = \rC^*(\iota(\A))$. In our arguments, we will be concerned with a handful of particular $\rC^*$-covers and we will frequently deal with several $\rC^*$-covers at once. For this reason, we will typically not record the Hilbert space that an operator algebra is represented upon.

The most notable $\rC^*$-cover we work with is the maximal $\rC^*$-cover. The \emph{maximal $\rC^*$-cover} of an operator algebra $\A$ is the essentially unique $\rC^*$-cover, denoted $(\rC^*_{max}(\A),\mu)$, satisfying the following universal property: whenever $\rho:\A\rightarrow B(\H_\rho)$ is a representation of $\A$, there is a unique $*$-representation $\theta:\rC^*_{max}(\A)\rightarrow B(\H_\rho)$ such that $\theta\circ\mu= \rho$. In particular, whenever $(\fA, \iota)$ is a $\rC^*$-cover for $\A$, there is a unique surjective $*$-representation $q_\fA:\rC^*_{max}(\A)\rightarrow\fA$ such that $q_\fA\circ\mu = \iota$. We will utilize this notation throughout. The maximal $\rC^*$-cover is known to exist for any operator algebra and can be constructed in a concrete way \cite[Proposition 2.4.2]{blecher2004operator}. That is, one constructs the completely isometric representation $\mu:\A\rightarrow\rC^*_{max}(\A)$ by taking an appropriately large direct sum of representations of $\A$.

We remark that to answer Question \ref{Q:main} it suffices to consider the unital case. Indeed, the maximal $\rC^*$-cover of an operator algebra preserves unitizations in a natural way \cite[2.4.3]{blecher2004operator}. With this in mind, we will occasionally make the assumption that the operator algebra $\A$ is unital.

Another $\rC^*$-cover we make use of is the \emph{$\rC^*$-envelope} of an operator algebra, denoted $(\rC^*_e(\A),\iota_{env})$. The $\rC^*$-envelope satisfies the following universal property: whenever $(\fA, \iota)$ is some $\rC^*$-cover of $\A$, there is a surjective $*$-representation $\pi:\fA\rightarrow \rC^*_e(\A)$ such that $\pi\circ\iota = \iota_{env}$. We refer to the kernel of the representation $q_{\rC^*_e(\A)}$ as the \emph{Silov ideal} of $\A$, denoted Sh$\A$. The Silov ideal is known to be the largest closed two-sided ideal of $\rC^*_{max}(\A)$ such that the quotient map $\rC^*_{max}(\A)\rightarrow \rC^*_{max}(\A)/\text{Sh}\A$ is completely isometric on $\mu(\A)$ \cite{arveson2008noncommutative}. It is not obvious, but this $\rC^*$-cover will always exist. Existence of the $\rC^*$-envelope was originally shown by Hamana \cite{hamana1979injective} through so-called injective envelopes. An alternative, dilation theoretic, proof was recently completed due to several significant contributions \cite{arveson2008noncommutative}, \cite{davidson2015choquet}, \cite{dritschel2005boundary}, \cite{muhly1998algebraic}.

\subsection{Residual finite-dimensionality} \label{SS:RFD}

An operator algebra $\A$ is \emph{residually finite-dimensional} (or RFD) if there are positive integers $r_\lambda, \lambda\in\Lambda,$ and a completely isometric representation $\iota:\A\rightarrow\prod_{\lambda\in\Lambda}\bM_{r_\lambda}$. That is, there is a completely norming family of representations for $\A$ on finite-dimensional Hilbert spaces.

For $\rC^*$-algebras there are many ways in which this property has been characterized. We recall some of the most relevant characterizations for our discussion. Let $\fA$ be a $\rC^*$-algebra and $\pi:\fA\rightarrow B(\H)$ be a $*$-representation. Then, $\pi$ is a \emph{residually finite-dimensional} (or RFD) $*$-representation if there is a net of (possibly degenerate) $*$-representations $\pi_\lambda: \fA\rightarrow B(\H)$ such that $\pi_\lambda(\fA)\H$ is finite-dimensional and \[\sot\lim_\lambda\pi_\lambda(t) = \pi(t), \ \ \ \ \ \ t\in\fA.\]Due to \cite[3.5.2]{diximier1977c}, this is the same as the topology of point-weak convergence. The following may be found in \cite{archbold1995residually}.

\begin{theorem}\label{T:RFDdensity}
Let $\fA$ be a $\rC^*$-algebra. Then, the following statements are equivalent: \begin{enumerate}[{\rm (i)}]
\item $\fA$ is an RFD $\rC^*$-algebra;
\item the collection of unitary equivalence classes of irreducible finite-dimensional $*$-representations in $\widehat{\fA}$ is dense;
\item every $*$-representation of $\fA$ is residually finite-dimensional.
\end{enumerate}
\end{theorem}

We analyze residual finite-dimensionality of operator algebras in the language of statements {\rm (ii)} and {\rm (iii)}. These statements are the impetus for Sections \ref{S:RFDmax} and \ref{S:RFDrepns}.

Residual finite-dimensionality for operator algebras was first studied in \cite{clouatre2019residually}. Therein, a complementary analysis discussing residual finite-dimensionality of $\rC^*$-covers was provided. Our motivation deals with the residual finite-dimensionality of the maximal $\rC^*$-cover. However, we remark that there is no known description of the RFD $\rC^*$-covers of an operator algebra. The following Example \ref{E:DiscAlg} is of great importance. Here, we illustrate that there may be several RFD $\rC^*$-covers of an operator algebra but there may still be intermediate $\rC^*$-covers which are not RFD.

\begin{example}\label{E:DiscAlg}
Let $\rA(\bD)$ denote the algebra of complex functions which are holomorphic on the open unit disk $\bD$ and continuous on the boundary $\bT$. The algebra $\rA(\bD)$ is called the disc algebra. There are several well-studied $\rC^*$-covers for the disc algebra. Indeed, there are obvious completely isometric representations $\iota: \rA(\bD)\rightarrow \rC(\bT)$ and $j: \rA(\bD)\rightarrow \rC(\ol{\bD})$ which produce RFD $\rC^*$-covers of $\rA(\bD)$.

The matrical von-Neumann inequality \cite[Corollary 3.12]{paulsen2002completely} implies that the maximal $\rC^*$-cover of $\rA(\bD)$ is the universal $\rC^*$-algebra generated by a contraction (alternatively, see \cite[Example 2.3]{blecher1999modules}). By \cite[Proposition 2.2]{courtney2018universal}, it follows that the maximal $\rC^*$-cover of $\rA(\bD)$ is RFD.

Despite the fact that $\rC^*_{max}(\rA(\bD))$ is RFD, there are other $\rC^*$-covers of $\rA(\bD)$ which fail to be RFD. To show this, we utilize standard facts on the Toeplitz algebra which may be found within \cite{agler2002pick}, \cite{arveson1998subalgebras}, \cite{douglas2012banach}. Let $H^2$ denote the classical Hardy space on the disc. Each $f\in \rA(\bD)$ determines a multiplication operator $M_f: H^2\rightarrow H^2$ such that $\lVert M_f\rVert = \lVert f\rVert.$ Further, the map \[\omega: \rA(\bD)\rightarrow B(H^2), \ \ \ \ \ \ f\mapsto M_f\]is a completely isometric representation. The $\rC^*$-algebra $\fT = \rC^*(\omega(\rA(\bD)))$ is the Toeplitz algebra and there is a short exact sequence \[0\rightarrow\fK(H^2)\rightarrow\fT\rightarrow \rC(\bT)\rightarrow 0.\]Note that residual finite-dimensionality passes to subalgebras and $\fT$ contains the ideal of compact operators $\fK(H^2)$. The $\rC^*$-algebra $\fK(H^2)$ is not RFD by Theorem \ref{T:RFDdensity} {\rm (ii)}. Hence, $\fT$ is not RFD either.
\end{example}

A very recent investigation showed that the maximal $\rC^*$-cover for the polydisc algebra $\rA(\bD^n)$ is also RFD \cite[Corollary 5.13]{clouatre2021finite}. Also, while we saw in Example \ref{E:DiscAlg} that residual finite-dimensionality does not pass to quotient algebras of $\rC^*_{max}(\A)$, residual finite-dimensionality clearly passes to ideals of $\rC^*_{max}(\A)$. Topologically, this statement may be interpreted as saying that residual finite-dimensionality passes to open subsets of the spectrum but does not necessarily pass to closed subsets.

\section{Topology and Ordering for $\rC^*$-covers}\label{S:TopOrder}

Within this section, we equate order theoretic statements on the collection of $\rC^*$-covers of an operator algebra with topological data from the spectrum of the maximal $\rC^*$-cover.

Fix an operator algebra $\A$. One can define a partial ordering on the $\rC^*$-covers of $\A$ as in \cite[Proposition 2.1.1]{hamidi2019admissibility}. If $(\fA, \iota)$ and $(\fB, j)$ are $\rC^*$-covers for $\A$, then we say $(\fA, \iota)\preceq(\fB, j)$ if and only if there is a surjective $*$-representation $\pi:\fB\rightarrow\fA$ such that $\pi\circ j = \iota$. Whenever we have $(\fA, \iota)\preceq(\fB, j)$ and $(\fB, j)\preceq(\fA,\iota)$, we say that the $\rC^*$-covers are \emph{equivalent}, denoted $(\fA, \iota)\sim(\fB, j)$. Two $\rC^*$-covers are equivalent if and only if there is a $*$-isomorphism $\pi:\fB\rightarrow\fA$ such that $\pi\circ j = \iota$. Up to equivalence of $\rC^*$-covers, the maximal $\rC^*$-cover is the unique maximal element under this ordering whereas the $\rC^*$-envelope is the unique minimal element \cite[Examples 2.1.8 and 2.1.13]{hamidi2019admissibility}. The reader should consult the thesis of Hamidi \cite[Chapters 1-2]{hamidi2019admissibility} for details. The fundamental approach we take in this section complements Hamidi's work. Here, we equate the partial ordering with a lattice structure determined by the spectrum of the maximal $\rC^*$-cover. These interpretations will be applied to residual finite-dimensionality in Section \ref{S:RFDmax}.

We start by identifying properties of the spectrum of the maximal $\rC^*$-cover in relation to the $\rC^*$-covers of $\A$. First, the spectrum of any $\rC^*$-cover of $\A$ may be identified as a closed subset of $\widehat{\rC^*_{max}(\A)}$: Let $\S(\fA, \iota)\subset\widehat{\rC^*_{max}(\A)}$ denote the collection of unitary equivalence classes consisting of irreducible $*$-representations $\pi:\rC^*_{max}(\A)\rightarrow B(\H_\pi)$ such that $\pi = \sigma\circ q_\fA$ where $\sigma$ is an irreducible $*$-representation of $\fA$. Theorem \ref{T:DiximierTopology} reveals that $\S(\fA, \iota)$ is naturally homeomorphic to $\widehat{\fA}$. In fact, every closed subset which contains the spectrum of the $\rC^*$-envelope is of this form.

\begin{theorem}\label{T:closcov}
Let $\A$ be an operator algebra and $\C\subset\widehat{\rC^*_{max}(\A)}$ be some subset. Then, $\C = \S(\fA, \iota)$ for some $\rC^*$-cover $(\fA, \iota)$ of $\A$ if and only if $\C$ is closed and contains $\S(\rC^*_e(\A), \iota_{env})$.
\end{theorem}

\begin{proof}
$(\Rightarrow)$ By Theorem \ref{T:DiximierTopology} {\rm (ii)}, the subset $\S(\fA, \iota)$ is closed. Since $q_\fA:\rC^*_{max}(\A)\rightarrow\fA$ is completely isometric on $\mu(\A)$, we have that $\ker q_\fA\subset \text{Sh}\A$. Then, Theorem \ref{T:DiximierTopology} implies that $\S(\rC^*_e(\A), \iota_{env})\subset\S(\fA, \iota)$.

$(\Leftarrow)$ As $\C$ is closed, by Theorem \ref{T:DiximierTopology}, $\C$ is the collection of all unitary equivalence classes consisting of irreducible $*$-representations of $\rC^*_{max}(\A)$ which vanish on some ideal $\fJ$. Similarly, the subset $\S(\rC^*_e(\A), \iota_{env})$ consists of all equivalence classes of irreducible $*$-representations of $\rC^*_{max}(\A)$ which vanish on the Silov ideal $\text{Sh}\A$. As $\S(\rC^*_e(\A), \iota_{env})\subset\C$, we have that $\fJ\subset \text{Sh}\A$ by Theorem \ref{T:DiximierTopology}. Let $\fA = \rC^*_{max}(\A)/\fJ$ and $q:\rC^*_{max}(\A)\rightarrow \fA$ be the corresponding quotient map. Define a representation of $\A$ by $\iota = q\circ\mu$. The map $\iota$ is completely isometric by \cite[Proposition 1.3.2]{hamidi2019admissibility}. Furthermore, \[\rC^*(\iota(\A)) = \rC^*(q\circ\mu(\A)) = q(\rC^*(\mu(\A))) = \fA.\]So $(\fA, \iota)$ is a $\rC^*$-cover of $\A$ and $q_\fA = q$ by uniqueness of the representation $q_\fA$. Therefore $\C$ is the collection of all unitary equivalence classes consisting of irreducible $*$-representations which vanish on $\fJ = \ker q_\fA$ and $\C = \S(\fA, \iota)$.
\end{proof}

It will be a consequence of the subsequent Proposition \ref{P:toporder} that the $\rC^*$-cover obtained in Theorem \ref{T:closcov} is unique up to equivalence. Also, despite the conclusion of Theorem \ref{T:closcov}, the relative topology of $\S(\rC^*_e(\A), \iota_{env})$ does not impose an obvious restriction regarding residual finite-dimensionality of $\rC^*$-covers. In fact, there are operator algebras whose $\rC^*$-envelope does not possess any finite-dimensional $*$-representations while there exist other $\rC^*$-covers which are RFD.

\begin{example}\label{E:cuntz}
Take a pair of isometries $V,W\in B(\H)$ satisfying $VV^*+WW^*=I$ and let $\M$ be the unital subspace of $B(\H)$ generated by $V$ and $W$. Form the unital operator algebra $\A$ consisting of elements of the form \[\begin{bmatrix} \lambda I & T\\ 0 & \mu I\end{bmatrix}\in B(\H^{(2)})\]where $\lambda,\mu\in\bC$ and $T\in\M$. We have that $\rC^*_e(\A)\cong\bM_2(\O_2)$ where $\O_2$ denotes the Cuntz algebra \cite[Example 5]{clouatre2020finite}. So $\rC^*_e(\A)$ is a simple, infinite-dimensional $\rC^*$-algebra. Consequently, the subset $\S(\rC^*_e(\A), \iota_{env})$ does not possess any equivalence classes of finite-dimensional irreducible $*$-representations. On the other hand, $\rC^*_{max}(\A)$ is RFD by \cite[Theorem 5.1]{clouatre2019residually}.
\end{example}

As there can easily be many inequivalent $\rC^*$-covers of a fixed operator algebra, Theorem \ref{T:closcov} indicates that the topology of $\widehat{\rC^*_{max}(\A)}$ is non-trivial. This is the case for even the simplest choice of non self-adjoint algebra.

\begin{example}\label{E:UTtopology}
Let $\T_2$ denote the algebra of upper-triangular $2\times 2$ matrices. We show that the spectrum of $\rC^*_{max}(\T_2)$ is not even Hausdorff. Let \[\fM = \{ f\in C([0,1], \bM_2): f(0) \text{ is a diagonal matrix}\}\]and let $\mu:\T_2\rightarrow\fM$ be defined by \[\begin{bmatrix} a & b\\ 0 & c\end{bmatrix}\mapsto \begin{bmatrix} aI & b\sqrt{\cdot}\\ 0 & cI\end{bmatrix}.\]In \cite[Example 2.4]{blecher1999modules}, it was shown that $\mu$ is completely isometric and that $(\fM, \mu)$ is equivalent to the maximal $\rC^*$-cover of $\T_2$.

Let $\iota:\T_2\rightarrow\bM_2$ be the identity representation. As $\bM_2$ is simple, $(\bM_2, \iota)$ is the $\rC^*$-envelope for $\T_2$. Let $\gamma_\xi:\fM\rightarrow\bM_2$ denote evaluation at $\xi\in(0,1]$.  Observe that Theorem \ref{T:closcov} implies that any closed subset of $\widehat{\fM}$ containing $[\gamma_1]$ determines the spectrum of a $\rC^*$-cover for $\T_2$. Indeed $\gamma_1\circ\mu = \iota$ and then it is easy to see that $\S(\bM_2, \iota) = \{[\gamma_1]\}$ as there is a unique irreducible $*$-representation of $\bM_2$.

Note that $\rC([0,1],\bM_2)\cong \rC([0,1])\otimes\bM_2$ is a liminal (or CCR) $\rC^*$-algebra \cite[Theorem 2 (c)]{tomiyama1967applications}. As a result, $\fM$ is also a liminal $\rC^*$-algebra \cite[Proposition 4.2.4]{diximier1977c}. As $\fM$ is unital, it follows that all irreducible $*$-representations of $\fM$ are finite-dimensional. Whence, $\widehat{\fM}$ is homeomorphic to $\text{Prim}(\fM)$ via the natural mapping \cite[Proposition 3.1.6]{diximier1977c}.

For $f\in\fM$ and $j=1,2$, define $\eta_j(f)$ to be the $(j,j)$-th entry of $f(0)$. Then $\eta_j:\fM\rightarrow\bC$ is a character. Take a sequence of points $(\xi_n)_{n\geq 1}\subset(0,1]$ converging to $0$. Let $\fJ_n=\ker\gamma_{\xi_n}$ and for $j=1,2$, let $\fL_j=\ker\eta_j$. Then, we see that $\bigcap_{n=1}^\infty \fJ_n\subset \fL_1, \fL_2$. Therefore, $\fL_1$ and $\fL_2$ are distinct accumulation points for the sequence $(\fJ_n)_{n\geq 1}$. In particular, $\widehat{\fM}$ is not Hausdorff.
\end{example}

The reader should compare the conclusion of Example \ref{E:UTtopology} with the following classical result of Kaplansky \cite[Theorem 4.2]{kaplansky1951structure}. Therein, it was shown that if all irreducible $*$-representations of a $\rC^*$-algebra are finite-dimensional and of the same dimension, then the spectrum is Hausdorff.

Now we showcase how the partial ordering on the $\rC^*$-covers of an operator algebra can be completely understood by topological statements.

\begin{proposition}\label{P:toporder}
Let $\A$ be an operator algebra and $(\fA, \iota), (\fB, j)$ be $\rC^*$-covers of $\A$. Then, the following statements are equivalent: \begin{enumerate}[{\rm (i)}]
\item $\S(\fA, \iota)$ is a subset of $\S(\fB, j)$;
\item $q_\fA$ is weakly contained in $q_\fB$;
\item $(\fA, \iota)\preceq(\fB, j).$
\end{enumerate}In particular, $\S(\fA, \iota)=\S(\fB, j)$ if and only if $(\fA, \iota)\sim(\fB, j)$.
\end{proposition}

\begin{proof}
{\rm(i)}$\Rightarrow${\rm(ii)}: As $\S(\fB, j)^C\subset\S(\fA, \iota)^C$, we have $\ker q_\fB\subset \ker q_\fA$ by Theorem \ref{T:DiximierTopology}.

{\rm (ii)}$\Rightarrow${\rm(iii)}: Since $q_\fA$ is weakly contained in $q_\fB$, there is a surjective $*$-representation $q:\fB\rightarrow\fA$ such that $q\circ q_\fB = q_\fA.$ We see that \[q\circ j = q\circ q_\fB\circ\mu= q_\fA\circ\mu = \iota.\]Therefore $(\fA,\iota)\preceq(\fB, j)$.

{\rm (iii)}$\Rightarrow${\rm (i)}: Let $q:\fB\rightarrow \fA$ be a surjective $*$-representation such that $q\circ j =\iota$. As $q\circ q_\fB\circ \mu = \iota,$ we have that $q\circ q_\fB = q_\fA$. Take $[\pi]\in\S(\fA, \iota)$ and let $\pi$ be a representative for $[\pi]$. Express $\pi = \sigma\circ q_\fA$ where $\sigma$ is an irreducible $*$-representation of $\fA$. Then $\pi = (\sigma\circ q)\circ q_\fB$ where $\sigma\circ q$ is an irreducible $*$-representation of $\fB$. So $[\pi]\in\S(\fB, j)$.
\end{proof}

Next, we provide a refinement of Proposition \ref{P:toporder}. First we recall the complete lattice structure of $\rC^*$-covers in \cite[Section 2.1]{hamidi2019admissibility}. Fix an operator algebra $\A$. Let $\text{Cov}(\A)$ denote the collection of all equivalence classes of $\rC^*$-covers of $\A$. Define a partial ordering $\preceq_q$ on $\text{Cov}(\A)$ by $[(\fA, \iota)]\preceq_q[(\fB, j)]$ if and only if $(\fA, \iota)\preceq(\fB, j)$. For brevity, we will refer solely to the ordering $\preceq$. A complete lattice structure on $\text{Cov}(\A)$ is defined as follows. Let $\C = \{[(\fA_\lambda, \iota_\lambda)]\}\subset\text{Cov}(\A)$ be some collection. Then, $\sup\C$ is the equivalence class $[(\rC^*(\iota(\A)), \iota)]$ where $\iota = \bigoplus_\lambda\iota_\lambda.$ For each $\lambda$, let $\fJ_\lambda = \ker q_{\fA_\lambda}$ and let $\fJ = \ol{\sum_\lambda\fJ_\lambda}$ be the norm closure of the ideal generated by $\bigcup_\lambda \fJ_\lambda$. Then, $\inf\C$ is the equivalence class $[(\rC^*_{max}(\A)/\fJ,q\circ\mu)]$ where $q:\rC^*_{max}(\A)\rightarrow\rC^*_{max}(\A)/\fJ$ is the quotient map.

We show that $\text{Cov}(\A)$ is isomorphic to a complete lattice arising naturally from $\widehat{\rC^*_{max}(\A)}$. To this end, let $\fQ(\A)$ denote the collection of all closed subsets of $\widehat{\rC^*_{max}(\A)}$ which contain $\S(\rC^*_e(\A), \iota_{env})$, ordered by inclusion. Then $\fQ(\A)$ is a complete lattice where $\inf\D = \bigcap_\lambda F_\lambda$ and $\sup\D = \ol{\bigcup_\lambda F_\lambda}$.

\begin{theorem}\label{T:LatticeIsom}
Let $\A$ be an operator algebra. Then the mapping\[\Omega: \text{Cov}(\A)\rightarrow\fQ(\A), \ \ \ \ \ \ [(\fA, \iota)]\mapsto\S(\fA, \iota),\]is an isomorphism of complete lattices.
\end{theorem}

The subsequent arguments may be found within \cite[Section 2.1]{hamidi2019admissibility} but we connect the details to draw different conclusions. Hence, we include the proof for completeness.

\begin{proof}
The map $\Omega$ is a well-defined bijection by Theorem \ref{T:closcov} and Proposition \ref{P:toporder}. We will show that $\Omega$ respects {\rm (i)} infima and {\rm (ii)} suprema. Let $\C = \{[(\fA_\lambda, \iota_\lambda)]\}\subset\text{Cov}(\A)$ be some collection. For each $\lambda$, let $\fJ_\lambda = \ker q_{\fA_\lambda}$. Let $\pi$ be an irreducible $*$-representation of $\rC^*_{max}(\A)$.

{\rm (i)} Let $\fJ=\ol{\sum_\lambda\fJ_\lambda}$ and let $q: \rC^*_{max}(\A)\rightarrow\rC^*_{max}(\A)/\fJ$ be the quotient map. Observe that $\S(\rC^*_{max}(\A)/\fJ, q\circ\mu)$ consists of the equivalence classes of irreducible $*$-representations of $\rC^*_{max}(\A)$ which vanish on $\fJ$. As $\pi$ vanishes on $\fJ$ if and only if $\pi$ vanishes on each $\fJ_\lambda$, we establish that \[\S(\rC^*_{max}(\A)/\fJ, q\circ\mu) = \bigcap_\lambda\S(\fA_\lambda, \iota_\lambda).\]

{\rm (ii)} Let $\iota = \bigoplus_\lambda\iota_\lambda$ and $\fA = \rC^*(\iota(\A))$. We show that \begin{equation}\S(\fA,\iota) = \ol{\bigcup_\lambda\S(\fA_\lambda, \iota_\lambda)}.\end{equation}Suppose that $\pi=\sigma\circ q_{\fA_\lambda}$ for some $\lambda$ and some irreducible $*$-representation $\sigma$ of $\fA_\lambda$. Let $\gamma_\lambda: \fA\rightarrow\fA_\lambda$ denote the obvious surjective $*$-representation. Then $\sigma\circ\gamma_\lambda$ is an irreducible $*$-representation of $\fA$. We have that \[\sigma\circ\gamma_\lambda\circ q_\fA\circ \mu = \sigma\circ \gamma_\lambda\circ \iota = \sigma\circ\iota_\lambda = \sigma\circ q_{\fA_\lambda}\circ \mu = \pi\circ\mu.\]So $\pi = \sigma\circ\gamma_\lambda\circ q_\fA$. Hence, $\S(\fA_\lambda, \iota_\lambda)\subset\S(\fA, \iota)$. By Theorem \ref{T:closcov}, $\S(\fA, \iota)$ is closed and so we have that $\ol{\bigcup_\lambda\S(\fA_\lambda, \iota_\lambda)}\subset\S(\fA, \iota).$

Conversely, by Theorem \ref{T:closcov}, there is a $\rC^*$-cover $(\fB, j)$ of $\A$ such that $\S(\fB, j)=\ol{\bigcup_\lambda\S(\fA_\lambda, \iota_\lambda)}$. For each $\lambda$, there is a surjective $*$-representation $\beta_\lambda:\fB\rightarrow \fA_\lambda$ such that $\beta_\lambda\circ j = \iota_\lambda$ by Proposition \ref{P:toporder}. Let $\beta = \bigoplus_\lambda\beta_\lambda$. Then $\beta$ is a $*$-representation satisfying $\beta\circ j = \iota$. So \[\beta(\fB) = \beta(\rC^*(j(\A))) = \rC^*(\beta\circ j(\A)) = \fA\]and we have that $(\fA, \iota)\preceq(\fB, j)$. By Proposition \ref{P:toporder}, we have that $\S(\fA, \iota)\subset\S(\fB, j).$ So equation (1) holds.
\end{proof}

Theorem \ref{T:LatticeIsom} confirms that the language arising from the spectral topology of the maximal $\rC^*$-cover is equivalent to the ordering for $\rC^*$-covers. This reflects the fact that, in place of studying completely isometric representations of an operator algebra, one may study a single completely isometric representation, namely the canonical embedding $\mu:\A\rightarrow\rC^*_{max}(\A)$, along with the representation theory of $\rC^*_{max}(\A)$. Due to the classical construction of the maximal $\rC^*$-cover, it is feasible to interpret the canonical embedding $\mu$ as encoding all the representation theory for the algebra $\A$.

Since the representations of $\A$ are in one-to-one correspondence with the $*$-representations of $\rC^*_{max}(\A)$, it is feasible to identify the spectrum of an operator algebra as the spectrum of the maximal $\rC^*$-cover. However, it is not clear to what degree one can identify the topology at the level of the operator algebra. We provide a partial answer to this question.

Let $\A$ be an operator algebra. If $\pi$ and $\rho$ are representations of $\A$, then $\pi$ and $\rho$ are \emph{approximately unitarily equivalent} if \[\pi(a) = \lim_{n\rightarrow\infty}U_n^*\rho(a)U_n, \ \ \ \ \ \ a\in\A,\]where $U_n:\H_\pi\rightarrow\H_\rho$ is a unitary operator for each $n\in\bN$.

\begin{proposition}\label{P:aueLift}
Let $\A$ be an operator algebra. Let $\pi$ and $\rho$ be representations of $\A$ and let $\theta_\pi, \theta_\rho$ be $*$-representations of $\rC^*_{max}(\A)$ satisfying $\theta_\pi\circ\mu=\pi$ and $\theta_\rho\circ\mu=\rho$. Then, $\pi$ and $\rho$ are approximately unitarily equivalent if and only if $\theta_\pi$ and $\theta_\rho$ are approximately unitarily equivalent.
\end{proposition}

\begin{proof}
$(\Rightarrow)$ Suppose that \[\pi(a)=\lim_{n\rightarrow\infty}U_n^*\rho(a)U_n, \ \ \ \ \ \ a\in\A,\]where $U_n:\H_\pi\rightarrow\H_\rho$ is unitary for each $n\in\bN$. So \[\theta_\pi(\mu(a)) = \lim_{n\rightarrow\infty}U_n^*\theta_\rho(\mu(a))U_n, \ \ \ \ \ \ a\in\A.\]As addition, multiplication and the adjoint are continuous in the norm topology, we infer that \[\theta_\pi(w) = \lim_{n\rightarrow\infty}U_n^*\theta_\rho(w)U_n\]for each $w\in\rC^*_{max}(\A)$ which lies in the linear span of words in $\mu(\A)\cup\mu(\A)^*$. This set is dense in $\rC^*_{max}(\A)$ and so $\theta_\pi$ is approximately unitarily equivalent to $\theta_\rho$.

$(\Leftarrow)$ Suppose that \[\theta_\pi(t) = \lim_{n\rightarrow\infty}U_n^*\theta_\rho(t)U_n, \ \ \ \ \ \ t\in\rC^*_{max}(\A),\]where $U_n:\H_\pi\rightarrow\H_\rho$ is unitary for each $n\in\bN$. In particular, for each $a\in\A$, \[\pi(a) = \theta_\pi(\mu(a)) = \lim_{n\rightarrow\infty}U_n^*\theta_\rho(\mu(a))U_n = \lim_{n\rightarrow\infty}U_n^*\rho(a)U_n.\]So $\pi$ and $\rho$ are approximately unitarily equivalent.
\end{proof}

By \cite[Corollary 4.1.10]{diximier1977c} and \cite[Corollary 2.5.6]{davidson1996c}, it is easily seen that irreducible $*$-representations of a separable $\rC^*$-algebra are weakly equivalent if and only if they are approximately unitarily equivalent. Consequently, the spectrum of a separable operator algebra can be identified pointwise as the spectrum of the maximal $\rC^*$-cover. In spite of this, it is not obvious whether one can make a similar statement involving weak containment (which determines the topology on $\widehat{\rC^*_{max}(\A)}$). Indeed, if one were to define weak containment of representations of $\A$ (in any of the equivalent forms presented in Subsection \ref{SS:topology}), then it is not true that this is equivalent to weak containment of the respective lifts to $\rC^*_{max}(\A)$. This can be seen by taking two completely isometric representations of an operator algebra which induce inequivalent $\rC^*$-covers.

\section{The RFD-Maximal $\rC^*$-Cover} \label{S:RFDmax}

Fix an RFD operator algebra $\A$. Our goal in this section is to analyse the largest RFD $\rC^*$-cover of $\A$ in the partial ordering among all $\rC^*$-covers of $\A$ which generate an RFD $\rC^*$-algebra. We refer to this $\rC^*$-cover as the \emph{RFD-maximal $\rC^*$-cover} of $\A$. It is unclear what likeness the RFD-maximal $\rC^*$-cover might have to the maximal $\rC^*$-cover. If Question \ref{Q:main} were to have an affirmative answer, then the RFD-maximal $\rC^*$-cover inherits all possible characteristics of the maximal $\rC^*$-cover. This will be the motivation behind the material found within Subsection \ref{SS:RFDmaxAlgProp}.

We outline our procedure. First we show that the RFD-maximal $\rC^*$-cover of an RFD operator algebra always exists (Theorem \ref{T:existRFDmax}). Following this, in Theorem \ref{T:topRFDmax}, we identify the spectrum of the RFD-maximal $\rC^*$-cover through spectral data of the maximal $\rC^*$-cover. In Theorem \ref{T:concreterepn}, we provide a concrete representation of the RFD-maximal $\rC^*$-cover which is akin to the classical construction of the maximal $\rC^*$-cover. This naturally endows the RFD-maximal $\rC^*$-cover with a certain universal property among RFD $\rC^*$-covers. Ending this section, we provide two instances in which the RFD-maximal $\rC^*$-cover shares commonality with the maximal $\rC^*$-cover (Theorems \ref{T:RFDmaxDirectSum} and \ref{T:freeproduct}).

\subsection{Abstract Characterization} \label{SS:abstract}

To determine the RFD-maximal $\rC^*$-cover, we require a pivotal observation on how residual finite-dimensionality is affected in $\text{Cov}(\A)$.

\begin{lemma}\label{L:supRFD}
Let $\A$ be an operator algebra and $\R = \{[(\fA_\lambda, \iota_\lambda)]\}\subset\text{Cov}(\A)$ be some collection such that $\fA_\lambda$ is an RFD $\rC^*$-algebra for each $\lambda$. If $(\fA, \iota)$ is any representative for $\sup\R$, then $\fA$ is an RFD $\rC^*$-algebra. 
\end{lemma}

\begin{proof}
First note that if $(\fB, j)$ is an RFD $\rC^*$-cover of $\A$, then any representative of $[(\fB, j)]$ is also RFD. Now, letting $\iota = \bigoplus_{\lambda}\iota_\lambda$, we have that $\sup\R = [(\rC^*(\iota(\A)), \iota)]$ and\[\rC^*(\iota(\A))\subset\prod_{\lambda}\fA_\lambda.\]As each $\fA_\lambda$ is an RFD $\rC^*$-algebra, $\rC^*(\iota(\A))$ is RFD as well. Therefore any representative of $\sup\R$ is RFD.
\end{proof}

Due to Lemma \ref{L:supRFD}, we may deduce the existence of the RFD-maximal $\rC^*$-cover:

\begin{theorem}\label{T:existRFDmax}
Let $\A$ be an RFD operator algebra and \[\R = \{ [(\fA, \iota)]\in\text{Cov}(\A): \fA \text{ is RFD}\}.\]Then, there is a unique element $[(\fR(\A), \upsilon)]\in\R$ which is maximal for $\R$.
\end{theorem}

\begin{proof}
As $\A$ is RFD, there is an RFD $\rC^*$-cover and so $\R$ is non-empty. Suppose that $\C$ is a chain in $\R$. By Lemma \ref{L:supRFD}, we have that $\sup\C\in\R$. By Zorn's Lemma, $\R$ contains maximal elements. If $[(\fS, \tau)]\in\R$ is also maximal for $\R$, then $(\fR, \upsilon)\preceq(\fS, \tau)$ and $(\fS, \tau)\preceq(\fR, \upsilon)$. So $[(\fR, \upsilon)]=[(\fS, \tau)]$.
\end{proof}

Given an RFD operator algebra $\A$, we will let $(\fR(\A), \upsilon)$ denote the \emph{RFD-maximal $\rC^*$-cover} of $\A$ as found in Theorem \ref{T:existRFDmax}. The RFD-maximal $\rC^*$-cover is unique up to equivalence of $\rC^*$-covers and, as with the maximal and minimal $\rC^*$-covers, we will frequently fix a representative of $[(\fR(\A), \upsilon)]$. Later, we will derive some explicit representations of this $\rC^*$-cover. First, we highlight the connection this $\rC^*$-cover has to Question \ref{Q:main}:

\begin{corollary}\label{C:RFDmaxiffmaxRFD}
Let $\A$ be an operator algebra. Then, $\rC^*_{max}(\A)$ is an RFD $\rC^*$-algebra if and only if $(\fR(\A), \upsilon)\sim(\rC^*_{max}(\A), \mu)$.
\end{corollary}

\begin{proof}
$(\Rightarrow)$ When $\rC^*_{max}(\A)$ is RFD, $\A$ is RFD itself. By Theorem \ref{T:existRFDmax}, we have that the RFD-maximal $\rC^*$-cover for $\A$ exists. Also, we have that $(\rC^*_{max}(\A), \mu)\preceq(\fR(\A),\upsilon)$ as $\rC^*_{max}(\A)$ is RFD. However, as the maximal $\rC^*$-cover is maximal in the ordering, we have $(\fR(\A), \upsilon)\preceq(\rC^*_{max}(\A), \mu)$.

$(\Leftarrow)$ If $(\fR(\A), \upsilon)\sim(\rC^*_{max}(\A), \mu)$, then there is a $*$-isomorphism $\pi:\fR(\A)\rightarrow\rC^*_{max}(\A)$. As $\fR(\A)$ is an RFD $\rC^*$-algebra, so is $\rC^*_{max}(\A)$.
\end{proof}

Next, we provide defining properties for the RFD-maximal $\rC^*$-cover of an operator algebra. This is achieved by analyzing the spectrum of the maximal $\rC^*$-cover. Here, we identify a dense subset of the spectrum of the RFD-maximal $\rC^*$-cover. This is very natural when utilizing Proposition \ref{P:toporder}: the RFD-maximal $\rC^*$-cover is the largest RFD $\rC^*$-cover of $\A$, both under the ordering of $\rC^*$-covers and under inclusion of the closed subsets $\S(\fA,\iota)$ where $(\fA, \iota)$ is an RFD $\rC^*$-cover of $\A$.

\begin{theorem}\label{T:topRFDmax}
Let $\A$ be an RFD operator algebra. Let $q: \rC^*_{max}(\A)\rightarrow\fR(\A)$ denote the unique surjective $*$-representation satisfying $q\circ\mu = \upsilon$ and let $\fJ = \ker q$. Then, the following statements hold: \begin{enumerate}[{\rm (i)}]
\item $\S(\fR(\A), \upsilon)$ is the closure of all equivalence classes which consist of finite-dimensional irreducible $*$-representations of $\rC^*_{max}(\A)$;
\item $\fJ$ does not possess any finite-dimensional $*$-representations;
\item $\S(\fR(\A), \upsilon)$ is the unique maximal closed subset of $\widehat{\rC^*_{max}(\A)}$ possessing a dense subset of equivalence classes of finite-dimensional irreducible $*$-representations.
\end{enumerate}
\end{theorem}

\begin{proof}
{\rm (i)}: In Theorem \ref{T:DiximierTopology} {\rm (iii)}, the homeomorphism produces a bijection between equivalence classes of finite-dimensional $*$-representations. Hence, by Theorem \ref{T:RFDdensity}, if $(\fA, \iota)$ is a $\rC^*$-cover of $\A$, then $\fA$ is RFD if and only if there is a dense subset of equivalence classes in $\S(\fA, \iota)$ consisting of finite-dimensional $*$-representations.

Let $\fF\subset\widehat{\rC^*_{max}(\A)}$ be the closure of all equivalence classes of finite-dimensional irreducible $*$-representations. As $\A$ is RFD, we have that $\fF$ contains $\S(\rC^*_e(\A), \iota_{env})$. Indeed, if $(\fA, \iota)$ be an RFD $\rC^*$-cover, then $\S(\fA, \iota)$ contains $\S(\rC^*_e(\A), \iota_{env})$ and possesses a dense subset of equivalence classes consisting of finite-dimensional $*$-representations. By Theorem \ref{T:closcov}, there is a $\rC^*$-cover $(\fA,\iota)$ such that $\S(\fA, \iota) = \fF$. Note that $\fF$ is the largest closed subset of $\widehat{\rC^*_{max}(\A)}$ containing a dense subset of equivalence classes of finite-dimensional irreducible $*$-representations. By Theorem \ref{T:LatticeIsom}, $(\fA,\iota)$ is the largest RFD $\rC^*$-cover up to equivalence. By Theorem \ref{T:existRFDmax}, we have that $(\fA, \iota)\sim(\fR(\A),\upsilon)$. Whence, $\S(\fR(\A), \upsilon) = \fF$ by Proposition \ref{P:toporder}.

{\rm (ii)}: Let $\pi$ be a finite-dimensional irreducible $*$-representation of $\fJ$. As $\fJ$ is an ideal of $\rC^*_{max}(\A)$, we may extend $\pi$ to a finite-dimensional irreducible $*$-representation of $\rC^*_{max}(\A)$, still denoted $\pi$. As $\pi\mid_\fJ$ is non-zero, we see that $[\pi]\notin\S(\fR(\A), \upsilon)$, a contradiction by {\rm (i)}.

{\rm (iii)}: This is obvious by {\rm (i)}.
\end{proof}

Using the language of Section \ref{S:TopOrder}, we may view Theorem \ref{T:topRFDmax} {\rm (iii)} as the topological counterpart to the order-theoretic statement of Theorem \ref{T:existRFDmax}.

Theorem \ref{T:topRFDmax} {\rm (ii)} is rather interesting with regards to Question \ref{Q:main}. A counterexample to Question \ref{Q:main} would be determined by the presence of a particular ideal whose irreducible $*$-representations are all infinite-dimensional. Thus, the ideal $\fJ$ should be easy to detect. In this sense, it is surprising that the answer to Question \ref{Q:main} is still unknown.

We highlight the following consequence which allows for a clear identification of the RFD-maximal $\rC^*$-cover:

\begin{corollary}\label{C:allFDfactor}
Let $\A$ be an RFD operator algebra and $(\fA, \iota)$ be a $\rC^*$-cover for $\A$. Then, the following statements are equivalent: \begin{enumerate}[{\rm (i)}]
\item $(\fA, \iota)\sim(\fR(\A), \upsilon)$;
\item $\fA$ is an RFD $\rC^*$-algebra and, if $[\sigma]\in\widehat{\rC^*_{max}(\A)}$ is an equivalence class consisting of finite-dimensional $*$-representations, then $[\sigma]\in\S(\fA, \iota)$.
\end{enumerate}
\end{corollary}

\begin{proof}
{\rm (i)}$\Rightarrow${\rm (ii)}: This follows by Theorem \ref{T:topRFDmax} {\rm (i)}.

{\rm (ii)}$\Rightarrow${\rm (i)}: Since $\fA$ is an RFD $\rC^*$-algebra, we have that $(\fA, \iota)\preceq(\fR(\A), \upsilon)$. By Proposition \ref{P:toporder}, we have that $\S(\fA, \iota)\subset\S(\fR(\A), \upsilon)$.  By Theorem \ref{T:closcov} and Theorem \ref{T:topRFDmax} {\rm(i)}, we see that  $\S(\fR(\A), \upsilon)\subset\S(\fA, \iota)$. Hence, $(\fA, \iota)\sim(\fR(\A), \upsilon)$ by Proposition \ref{P:toporder}.
\end{proof}

\subsection{Concrete Characterization} \label{SS:concrete}

We exhibit an explicit identification of the RFD-maximal $\rC^*$-cover. The maximal $\rC^*$-cover can be constructed in a concrete way and this is a common method of proving existence. We show that the RFD-maximal $\rC^*$-cover can be constructed through analogous methods. The proof is an adaptation of \cite[Proposition 2.4.2]{blecher2004operator}.

\begin{theorem}\label{T:concreterepn}
Let $\A$ be an RFD operator algebra. Then, there is an RFD $\rC^*$-cover $(\fQ, \iota)$ of $\A$ satisfying the following: for any representation $\rho:\A\rightarrow B(\H_\rho)$ on a finite-dimensional Hilbert space, there is a unique $*$-representation $\theta:\fQ\rightarrow B(\H_\rho)$ such that $\theta\circ\iota = \rho$. In particular, $(\fQ, \iota)$ is minimal among all $\rC^*$-covers of $\A$ which satisfy this property. Moreover, $(\fQ, \iota)\sim(\fR(\A), \upsilon)$.
\end{theorem}

\begin{proof}
Let $\F$ be the set of representations $\rho: \A\rightarrow\bM_r, r\in\bN$. As $\A$ is RFD, the map $\iota = \bigoplus_{\rho\in\F}\rho$ defines a completely isometric representation of $\A$ on a Hilbert space $\H=\bigoplus_{\rho\in\F}\bC^{r_\rho}$. Letting $\fQ = \rC^*(\iota(\A))$, we see that $\fQ$ is an RFD $\rC^*$-algebra and that $(\fQ, \iota)$ is a $\rC^*$-cover of $\A$.

Let $\rho:\A\rightarrow B(\H_\rho)$ denote a representation of $\A$ on a finite-dimensional Hilbert space. Then there is some $r_\chi\in\bN$ and a unitary operator $U:\bC^{r_\chi}\rightarrow\H_\rho$ such that $\chi = U^*\rho(\cdot) U\in\F$. Let $P_\chi$ denote the orthogonal projection of $\H$ onto $\bC^{r_\chi}$. By construction of $\fQ$, we obtain a $*$-representation $\pi:\fQ\rightarrow\bM_{r_\chi}$ defined by \[\pi(T) = P_\chi T\mid_{\bC^{r_\chi}}.\]Then, this dictates that \[\pi\circ\iota(a)  = \chi(a) = U^*\rho(a)U, \ \ \ \ \ \ a\in\A.\]Take $\theta = U\pi(\cdot)U^*:\fQ\rightarrow B(\H_\rho)$. Then $\theta$ is a finite-dimensional $*$-representation such that $\theta\circ\iota = \rho$.

For uniqueness of the map $\theta$, suppose that $\theta$ and $\tau$ are $*$-representations of $\fQ$ such that $\theta\circ\iota =\rho= \tau \circ\iota$. As $\tau$ and $\theta$ are $*$-representations which agree on $\iota(\A)$, we infer that $\tau$ and $\theta$ agree on the linear span of words in $\iota(\A)\cup\iota(\A)^*$. This latter set is dense within $\fQ$ and so $\theta = \tau$.

For minimality, suppose that $(\fA, j)$ is a $\rC^*$-cover of $\A$ such that every representation of $\A$ on a finite-dimensional Hilbert space lifts to $\fA$. For each $\rho\in\F$, we obtain a $*$-representation $\theta_\rho:\fA\rightarrow B(\H_\rho)$ such that $\theta_\rho\circ j = \rho$. In particular, we have that \[\iota = \bigoplus_{\rho\in\F}\rho = \bigoplus_{\rho\in\F}\theta_\rho\circ j.\]Then, $\Theta = \bigoplus_{\rho\in\F}\theta_\rho$ is a $*$-representation of $\fA$ such that $\Theta\circ j = \iota$. We see that \[\Theta(\fA) = \Theta(\rC^*(j(\A))) = \rC^*(\Theta\circ j(\A))  = \rC^*(\iota(\A)) = \fQ\]and so $(\fQ, \iota)\preceq(\fA, j)$ as desired.

Finally, we show that $(\fQ, \iota)\sim(\fR(\A), \upsilon)$. Since $\fQ$ is an RFD $\rC^*$-algebra, we show that the conditions in Corollary \ref{C:allFDfactor} hold. To this end, let $\pi: \rC^*_{max}(\A)\rightarrow B(\H_\pi)$ be a finite-dimensional irreducible $*$-representation. Then $\pi\circ\mu$ is a representation of $\A$ on a finite-dimensional Hilbert space. Hence, there exists a $*$-representation $\theta:\fQ\rightarrow B(\H_\pi)$ such that $\theta\circ\iota = \pi\circ\mu$. Then, $\theta\circ q_\fQ$ is a $*$-representation of $\rC^*_{max}(\A)$ satisfying \[\theta\circ q_\fQ\circ\mu = \theta\circ\iota = \pi\circ\mu.\]So $\theta\circ q_\fQ = \pi$ and we have that $[\pi]\in\S(\fQ, \iota)$ as desired.
\end{proof}

Theorem \ref{T:concreterepn} identifies a categorical property of the RFD-maximal $\rC^*$-cover. This property may be identified as a universal property of the RFD-maximal $\rC^*$-cover among the RFD $\rC^*$-covers of $\A$. However, it will only be a universal property of all $\rC^*$-covers of $\A$ if $\rC^*_{max}(\A)$ is RFD.

\begin{remark}\label{R:ConcreteRFDunital}
If $\A$ is a unital RFD operator algebra, then a slight modification to the proof of Theorem \ref{T:concreterepn} shows that the embedding $\iota$ can taken to be unital. In turn, the universal property of Theorem \ref{T:concreterepn} is equivalent to the following: for any unital completely contractive representation $\rho:\A\rightarrow B(\H_\rho)$ on a finite-dimensional Hilbert space, there is a unique (unital) $*$-representation $\theta:\fQ\rightarrow B(\H_\rho)$ such that $\theta\circ\iota=\rho.$ 
\end{remark}

\subsection{Preservation of Algebraic Properties}\label{SS:RFDmaxAlgProp}

We conclude this section by supplying a few instances in which the RFD-maximal $\rC^*$-cover is akin to the traditional maximal $\rC^*$-cover.

The maximal $\rC^*$-cover of an operator algebra is known to preserve certain algebraic constructions. See \cite[Proposition 2.2]{blecher1999modules}, \cite[2.4.3]{blecher2004operator}, \cite[Theorem 5.2]{clouatre2019residually}, \cite[Theorem 4.1]{duncan2004universal} for example. We provide supporting evidence for Question \ref{Q:main} to have an affirmative answer by showing that the RFD-maximal $\rC^*$-cover preserves a few of the same algebraic constructions (Theorems \ref{T:RFDmaxDirectSum} and \ref{T:freeproduct}). Indeed, for Question \ref{Q:main} to be true, the subsequent Theorems \ref{T:RFDmaxDirectSum} and \ref{T:freeproduct} have to hold. Withstanding this, the results herein provide useful facts to compute the RFD-maximal $\rC^*$-cover. First we address a minor albeit worthwhile point.

\begin{proposition}\label{P:RFDmaxUnit}
Let $\A$ be a non-unital RFD operator algebra and let $\widetilde{\A}$ denote the unitization of $\A$. Let $(\fR(\widetilde{\A}), \upsilon_1)$ be the RFD-maximal $\rC^*$-cover of $\widetilde{\A}$. If $\iota = \upsilon_1\mid_\A$ and $\fA = \rC^*(\iota(\A))$, then we have that $(\fA, \iota)\sim(\fR(\A), \upsilon).$
\end{proposition}

\begin{proof}
First note that $\fA$ is an RFD $\rC^*$-algebra. Let $\rho:\A\rightarrow B(\H)$ be a representation of $\A$ on a finite-dimensional Hilbert space. Let $\rho^+: \widetilde{\A}\rightarrow B(\H)$ be the unital representation which extends $\rho$. By Theorem \ref{T:concreterepn}, there is a $*$-representation $\theta:\fR(\widetilde{\A})\rightarrow B(\H)$ such that $\theta\circ\upsilon_1 = \rho^+$. Then $\sigma:=\theta\mid_\fA$ is a $*$-representation of $\fA$ such that $\sigma \circ\iota = \rho$. So $(\fA, \iota)\sim(\fR(\A), \upsilon)$ by Theorem \ref{T:concreterepn}.
\end{proof}

We proceed with a technical fact which will aid in our derivations for the remainder of this subsection. The following roughly states that if all finite-dimensional irreducible $*$-representations of an \emph{isomorphic} copy of the maximal $\rC^*$-cover factor through a fixed RFD $\rC^*$-cover, then this RFD $\rC^*$-cover is isomorphic to the RFD-maximal $\rC^*$-cover.

\begin{lemma}\label{L:fdfactorisomorphism}
Let $\A$ be an RFD operator algebra. Suppose $(\fR, \hat{\upsilon})$ and $(\fM, \hat{\mu})$ are $\rC^*$-covers of $\A$ where $(\fR, \hat{\upsilon})\preceq(\fM, \hat{\mu})$ and $\fM\cong\rC^*_{max}(\A)$. Let $q:\fM\rightarrow \fR$ be the surjective $*$-representation satisfying $q\circ\hat{\mu} = \hat{\upsilon}$. If every finite-dimensional $*$-representation $\pi$ of $\fM$ is of the form $\pi=\sigma\circ q$ for some $*$-representation $\sigma$ of $\fR$, then $\fR\cong\fR(\A)$.
\end{lemma}

\begin{proof}
Let $\F_{max}\subset\widehat{\rC^*_{max}(\A)}$ and $\F_\fM\subset\widehat{\fM}$ denote the unitary equivalence classes of finite-dimensional $*$-representations in the respective spectra. Let $\chi: \rC^*_{max}(\A)\rightarrow\fM$ be a $*$-isomorphism. Then the mapping \[\widehat{\fM}\rightarrow\widehat{\rC^*_{max}(\A)}, \ \ \ \ [\beta]\mapsto[\beta\circ\chi],\]is a homeomorphism which maps $\F_\fM$ bijectively onto $\F_{max}$. Let $\F_\fM'$ be a set of $*$-representations of $\fM$ such that $\F_\fM = \{ [\beta]: \beta\in\F_\fM'\}$. Then $\F_{max}' = \{\beta\circ\chi: \beta\in\F_\fM'\}$ is a complete system of representatives for $\F_{max}$.

For each $\beta\in\F_\fM'$, there is a $*$-representation $\beta_\fR$ of $\fR$ satisfying $\beta=\beta_\fR\circ q$. Conversely, if $\sigma$ is a finite-dimensional irreducible $*$-representation of $\fR$, then $\sigma\circ q$ is unitarily equivalent to an element of $\F_\fM'$. Whence, $\sigma$ is unitarily equivalent to $\beta_\fR$ for some $\beta\in\F_\fM'$. So $\{ [\beta_\fR]\in\widehat{\fR} : \beta\in\F_\fM'\}$ is the collection of all unitary equivalence classes of finite-dimensional irreducible $*$-representations of $\fR$. As $\fR$ is an RFD $\rC^*$-algebra, \[\fR\cong \left(\bigoplus_{\beta\in\F_\fM'}\beta_\fR\right)(\fR) = \left(\bigoplus_{\beta\in\F_\fM'}\beta\right)(\fM).\]

For each $\beta\in\F_\fM'$, we have that $[\beta\circ\chi]\in\F_{max}$. Conversely, if $\sigma$ is a finite-dimensional irreducible $*$-representation of $\rC^*_{max}(\A)$, then $\sigma$ is unitarily equivalent to $\beta\circ\chi$ for some $\beta\in\F_\fM'$. By Theorem \ref{T:RFDdensity} {\rm (ii)}, the collection of unitary equivalence classes of irreducible finite-dimensional $*$-representations in $\widehat{\fR(\A)}$ is dense. Then, by Theorem \ref{T:DiximierTopology} {\rm (iii)}, we see that the homeomorphism $\S(\fR(\A), \upsilon)\rightarrow \widehat{\fR(\A)}$ maps equivalence classes of finite-dimensional $*$-representations bijectively onto themselves. In turn, we may infer that\[\fR(\A)\cong \left(\bigoplus_{\theta\in\F_{max}'}\theta\right)(\rC^*_{max}(\A)).\]Then,\[\left(\bigoplus_{\theta\in\F_{max}'}\theta\right)(\rC^*_{max}(\A)) = \left(\bigoplus_{\beta\in\F_\fM'}\beta\circ\chi\right)(\rC^*_{max}(\A)) = \left(\bigoplus_{\beta\in\F_\fM'}\beta\right)(\fM)\cong\fR.\]
\end{proof}

Our first consequence will pertain countable direct sums of operator algebras. For unital operator algebras $\A_n, n\in\bN$, define \[\A = \bigoplus_{n=1}^\infty\A_n = \left\{(a_n)\in\prod_{n=1}^\infty\A_n : \lim_{n\rightarrow\infty}\lVert a_n\rVert = 0\right\}.\]In \cite[Theorem 5.2]{clouatre2019residually}, it was shown that\[\rC^*_{max}(\A)\cong\bigoplus_{n=1}^\infty\rC^*_{max}(\A_n).\]We provide supporting evidence for Question \ref{Q:main} to have an affirmative answer by showing that the RFD-maximal $\rC^*$-cover also satisfies this property.

\begin{theorem}\label{T:RFDmaxDirectSum}
For each $n\in\bN$, let $\A_n$ be a unital RFD operator algebra and let $\A = \bigoplus_{n=1}^\infty\A_n$. Then, $\fR(\A)\cong\bigoplus_{n=1}^\infty\fR(\A_n).$
\end{theorem}

\begin{proof}
For each $n\in\bN$, let $\mu_n$ and $\upsilon_n$ denote the completely isometric representations of $\A_n$ into $\rC^*_{max}(\A_n)$ and $\fR(\A_n)$, respectively. Moreover, let $q_n:\rC^*_{max}(\A_n)\rightarrow\fR(\A_n)$ denote the surjective $*$-representation satisfying $q_n\circ\mu_n=\upsilon_n$. Let \[\fR = \bigoplus_{n=1}^\infty \fR(\A_n), \ \ \ \ \ \fM = \bigoplus_{n=1}^\infty \rC^*_{max}(\A_n).\]For $n\in\bN$, let $\iota_n:\rC^*_{max}(\A_n)\rightarrow \fM$ denote the obvious embedding. We verify the conditions of Lemma \ref{L:fdfactorisomorphism}.

First note that $\fR$ is an RFD $\rC^*$-algebra. Define completely isometric representations of $\A$ by \[\hat{\upsilon}: (a_n)_{n\geq1}\mapsto (\upsilon_n(a_n))_{n\geq1}, \ \ \ \ \ \hat{\mu}: (a_n)_{n\geq1}\mapsto(\mu_n(a_n))_{n\geq1}.\]We show that $\rC^*(\hat{\mu}(\A))=\fM$. It is clear that $\rC^*(\hat{\mu}(\A))\subset\fM$. For the reverse inclusion, note that \[\iota_n(\rC^*_{max}(\A_n))\subset\rC^*(\hat{\mu}(\A)), \ \ \ \ \ n\in\bN.\]Upon taking finite sums, \[\bigoplus_{n=1}^m \rC^*_{max}(\A_n)\oplus0\subset \rC^*(\hat{\mu}(\A)), \ \ \ \ \ m\in\bN.\]For $t = (t_n)_{n\geq1}\in\fM$, we let \[s^{(m)}= (t_1, t_2, \ldots, t_m, 0, 0, \ldots)\in \rC^*(\hat{\mu}(\A)), \ \ \ \ \ m\in\bN.\]As $\lVert t_n\rVert$ tends to $0$, it follows that $s^{(m)}$ converges to $t$ in the norm topology of $\fM$. Whence, we have that $t\in\rC^*(\hat{\mu}(\A))$. Therefore $\fM = \rC^*(\hat{\mu}(\A))$. Similarly, we may establish that $\fR = \rC^*(\hat{\upsilon}(\A))$.

Define a $*$-representation $Q: \fM\rightarrow\fR$ by $Q((t_n)) = (q_n(t_n)).$ Note that $Q\circ\hat{\mu} = \hat{\upsilon}$. In particular, $Q$ is surjective as \[Q(\fM) = Q(\rC^*(\hat{\mu}(\A))) = \rC^*(Q\circ\hat{\mu}(\A)) = \rC^*(\hat{\upsilon}(\A)) = \fR.\]Therefore $(\fR, \hat{\upsilon})$ and $(\fM, \hat{\mu})$ are $\rC^*$-covers for $\A$ such that $(\fR, \hat{\upsilon})\preceq(\fM, \hat{\mu})$.

Let $\pi:\fM\rightarrow B(\H_\pi)$ be an irreducible finite-dimensional $*$-representation. Let $n\in\bN$ and note that the $\rC^*$-algebra $\iota_n(\rC^*_{max}(\A_n))$ is an ideal of $\fM$. Hence, the representation $\pi\mid_{\iota_n(\rC^*_{max}(\A_n))}$ is either identically zero or irreducible (Subsection \ref{SS:topology}). In particular, $\pi\circ\iota_n$ is either identically zero or an irreducible $*$-representation of $\rC^*_{max}(\A_n)$. In the latter case, Theorem \ref{T:topRFDmax} yields that $\pi\circ\iota_n = \sigma_n\circ q_n$ where $\sigma_n$ is a finite-dimensional irreducible $*$-representation of $\fR(\A_n)$. In the former case, one can take $\sigma_n$ to be the zero representation in order to establish that $\pi\circ\iota_n = \sigma_n\circ q_n$.

For each $n\in\bN$, let $\fJ_n = \ker q_n\subset \rC^*_{max}(\A_n)$. Then, $\fJ = \bigoplus_{n=1}^\infty \fJ_n=\ker Q$ is a closed two-sided ideal of $\fM$. Let $t = (t_n)_{n\geq 1}\in\fJ$ and \[s^{(m)} = (t_1, t_2, \ldots, t_m, 0, 0,\ldots), \ \ \ \ \ m\in\bN.\]Then \begin{align*}
\pi(t) & = \lim_{m\rightarrow\infty}\pi(s^{(m)})\\
& = \lim_{m\rightarrow\infty}\sum_{n=1}^m\pi\circ\iota_n(t_n)\\
& = \lim_{m\rightarrow\infty}\sum_{n=1}^m\sigma_n\circ q_n(t_n)\\
& = 0
\end{align*}as $t_n\in\fJ_n = \ker q_n$ for each $n\in\bN$. So $\pi$ vanishes on $\fJ$. Whence, $\pi=\sigma\circ Q$ where $\sigma$ is a finite-dimensional $*$-representation of $\fR$. By \cite[Theorem 5.2]{clouatre2019residually}, we have that $\rC^*_{max}(\A)\cong\fM$. Hence, Lemma \ref{L:fdfactorisomorphism} yields that $\fR\cong\fR(\A)$.
\end{proof}

Now we provide a similar statement for the unital free product of operator algebras. We recall the relevant details: Let $\A, \B$ be unital operator algebras. The \emph{free product} of $\A$ and $\B$, denoted $\A*\B$, is the unique operator algebra satisfying the following universal property:\begin{enumerate}[{\rm (i)}]
\item there exist unital completely isometric representations $\iota:\A\rightarrow\A*\B$ and $j:\B\rightarrow\A*\B$ such that $\iota(\A), j(\B)$ generate $\A*\B$ and if $\pi:\A\rightarrow\C$ and $\theta:\B\rightarrow\C$ are representations into a common operator algebra $\C$, then there exists a unique representation $\pi*\theta:\A*\B\rightarrow\C$ such that $(\pi*\theta)\circ\iota = \pi$ and $(\pi*\theta)\circ j = \theta$.
\end{enumerate}In the self-adjoint setting, this coincides with the usual notion of free products for $\rC^*$-algebras. In \cite[Theorem 4.1]{blecher1991explicit}, it was shown that such an object exists for any pair of operator algebras and that the norm of $X\in\bM_n(\A*\B)$ may be defined by\begin{equation}\lVert X\rVert = \inf\{ \lVert X_1\rVert \lVert X_2\rVert\ldots\lVert X_k\rVert\}.\end{equation}The infimum is over all expressions $X= X_1 X_2\ldots X_k$ where $X_1$ is an $n\times m_1$ matrix in $\iota(\A)$, $X_2$ is an $m_1\times m_2$ matrix in $j(\B)$, $X_3$ is an $m_2\times m_3$ matrix in $\iota(\A)$, etc.

\begin{theorem}\label{T:freeproduct}
Let $\A,\B$ be unital RFD operator algebras. Then, $\fR(\A*\B)\cong\fR(\A)*\fR(\B)$.
\end{theorem}

\begin{proof}
Let $(\rC^*_{max}(\A), \mu_\A)$ and $(\fR(\A), \upsilon_\A)$ denote the maximal and RFD-maximal $\rC^*$-covers of $\A$, respectively. Let $q_\A: \rC^*_{max}(\A)\rightarrow\fR(\A)$ be the surjective unital $*$-representation satisfying $q_\A\circ\mu_\A=\upsilon_\A$. Let $r_\A: \fR(\A)\rightarrow \fR(\A)*\fR(\B)$ and $m_\A:\rC^*_{max}(\A)\rightarrow\rC^*_{max}(\A)*\rC^*_{max}(\B)$ be the usual unital isometric $*$-representations. Let $\iota_\A:\A\rightarrow \A*\B$ be the unital completely isometric map given in {\rm (i)}. Define representations of $\B$ and $*$-representations of $\rC^*_{max}(\B)$ and $\fR(\B)$, denoted $\mu_\B, \upsilon_\B, q_\B, r_\B, m_\B$ and $\iota_\B$, in an analogous way.

We check the conditions of Lemma \ref{L:fdfactorisomorphism}. First note that $\fR(\A)*\fR(\B)$ is an RFD $\rC^*$-algebra by \cite[Theorem 3.2]{exel1992finite} and that $\rC^*_{max}(\A*\B)\cong \rC^*_{max}(\A)*\rC^*_{max}(\B)$ by \cite[Proposition 2.2]{blecher1999modules}. Take \[q=(r_\A\circ q_\A)*(r_\B\circ q_\B): \rC^*_{max}(\A)*\rC^*_{max}(\B)\rightarrow \fR(\A)*\fR(\B).\] Then $q$ is a surjective unital $*$-representation. Indeed, as $q_\A$ and $q_\B$ are surjective, we have that \[r_\A(\fR(\A))\cup r_\B(\fR(\B ))\subset q(\rC^*_{max}(\A)*\rC^*_{max}(\B)).\]The former set generates $\fR(\A)*\fR(\B)$ and so $q$ is surjective.

Define a representation \[\mu_*=(m_\A\circ\mu_\A)*(m_\B\circ\mu_\B):\A*\B\rightarrow\rC^*_{max}(\A)*\rC^*_{max}(\B).\]Note that $\mu_*\circ\iota_\A$ and $\mu_*\circ\iota_\B$ are completely isometric. By virtue of equation (2), $\mu_*$ is easily seen to be completely isometric. Similarly, the map\[\upsilon_* = (r_\A\circ\upsilon_\A)*(r_\B\circ\upsilon_\B):\A*\B\rightarrow\fR(\A)*\fR(\B)\]is also a completely isometric representation. Next, note that $\rC^*(\mu_*(\A*\B)) = \rC^*_{max}(\A)*\rC^*_{max}(\B)$: Indeed, by definition of $\mu_*$, we see that \[\rC^*(\mu_*\circ \iota_\A(\A)) = \rC^*(m_\A\circ\mu_\A(\A)) = m_\A(\rC^*_{max}(\A)) \subset \rC^*(\mu_*(\A*\B)).\]Similarly, we obtain that $m_\B(\rC^*_{max}(\B))\subset\rC^*(\mu_*(\A*\B))$. Therefore $$\rC^*(\mu_*(\A*\B))=\rC^*_{max}(\A)*\rC^*_{max}(\B).$$Observe that \[q\circ\mu_*\circ \iota_\A = q\circ m_\A\circ\mu_\A= r_\A\circ q_\A\circ \mu_\A = r_\A\circ\upsilon_\A.\]Similarly, we obtain $q\circ\mu_*\circ \iota_\B= r_\B\circ\upsilon_\B.$ By the universal property of the free product, we have that $q\circ\mu_* = \upsilon_*$. Whence, \[\rC^*(\upsilon_*(\A*\B)) = q(\rC^*(\mu_*(\A*\B))) = q(\rC^*_{max}(\A)*\rC^*_{max}(\B)) = \fR(\A)*\fR(\B).\]So $(\fR(\A)*\fR(\B), \upsilon_*)\preceq(\rC^*_{max}(\A)*\rC^*_{max}(\B), \mu_*)$.

Let $\pi: \rC^*_{max}(\A)*\rC^*_{max}(\B)\rightarrow B(\H_\pi)$ be a finite-dimensional irreducible $*$-representation. Then the universal property of the free product dictates that $\pi=\sigma_\A*\sigma_\B$ where $\sigma_\A: \rC^*_{max}(\A)\rightarrow B(\H_\pi)$ and $\sigma_\B: \rC^*_{max}(\B)\rightarrow B(\H_\pi)$ are $*$-representations. By Theorem \ref{T:topRFDmax} {\rm (i)}, there are $*$-representations $\hat{\sigma}_\A:\fR(\A)\rightarrow B(\H_\pi)$ and $\hat{\sigma}_\B:\fR(\B)\rightarrow B(\H_\pi)$ such that $\hat{\sigma}_\A\circ q_\A = \sigma_\A$ and $\hat{\sigma}_\B\circ q_\B =\sigma_\B$. Define a $*$-representation of $\fR(\A)*\fR(\B)$ by  $\Pi = \hat{\sigma}_\A*\hat{\sigma}_\B$. Observe that \[\Pi\circ q \circ m_\A = \Pi\circ r_\A\circ q_\A = \hat{\sigma}_\A\circ q_\A = \sigma_\A\]and similarly, $\Pi\circ q\circ m_\B = \sigma_\B$. So $\Pi\circ q = \pi$. Therefore, by Lemma \ref{L:fdfactorisomorphism}, we obtain that $\fR(\A*\B)\cong \fR(\A)*\fR(\B)$.
\end{proof}

\section{Residually Finite-Dimensional Representations} \label{S:RFDrepns}

In this section, we pursue a finer characterization of the RFD-maximal $\rC^*$-cover. Indeed, Theorem \ref{T:topRFDmax} only characterized the RFD-maximal $\rC^*$-cover up to a dense subset of representations. Now we uncover a wider class of representations by taking appropriate pointwise limits of representations. This is accomplished by analysing the so-called RFD and $*$-RFD representations of an operator algebra. Motivated by Question \ref{Q:main}, we consider the intermediate statement of whether these two classes of representations coincide.

The following two definitions will be central to our arguments. Let $\rho:\A\rightarrow B(\H)$ be a representation of an operator algebra $\A$. Then, $\rho$ is \emph{residually finite-dimensional} (or RFD) if there is a net of representations $\rho_\lambda:\A\rightarrow B(\H)$ such that $\rC^*(\rho_\lambda(\A))\H$ is finite-dimensional and \[\sot\lim_\lambda\rho_\lambda(a) = \rho(a), \ \ \ \ \ \ a\in\A.\]We call the representation \emph{$*$-residually finite-dimensional} (or $*$-RFD) if we further impose that \[\sot\lim_\lambda\rho_\lambda(a)^* = \rho(a)^*, \ \ \ \ \ \ a\in\A.\]In other words, $\rho$ is RFD (respectively, $*$-RFD) if it is the point-strong limit (or point-strong$*$ limit) of certain finite-dimensional representations of the operator algebra. These two concepts were introduced in \cite{clouatre2021finite} using different terminology. We will utilize the following result from their work several times. That is, $\rC^*_{max}(\A)$ is RFD if and only if every representation of $\A$ is $*$-RFD \cite[Theorem 3.3]{clouatre2021finite}. 

In the self-adjoint setting, RFD and $*$-RFD representations coincide and agree with the notion of a residually finite-dimensional $*$-representation (as introduced in Subsection \ref{SS:RFD}) . We warn the reader that these refer to three different notions. An RFD or $*$-RFD representation refers to a possibly non self-adjoint setting. On the other hand, an RFD $*$-representation refers to the self-adjoint setting.

First, we remark that residual finite-dimensionality of an operator algebra is equivalent to the algebra possessing either an RFD or $*$-RFD embedding: 

\begin{proposition}\label{P:cisRFD}
Let $\A$ be an operator algebra. Then, the following statements are equivalent:\begin{enumerate}[{\rm (i)}]
\item $\A$ is RFD;
\item there exists a completely isometric $*$-RFD representation of $\A$;
\item there exists a completely isometric RFD representation of $\A$.
\end{enumerate}
\end{proposition}

\begin{proof}
{\rm (i)}$\Rightarrow${\rm (ii)}: Let $\iota:\A\rightarrow\prod_\lambda\bM_{r_\lambda}$ be a completely isometric representation. For each $\lambda$, let $\gamma_\lambda: \rC^*(\iota(\A))\rightarrow\bM_{r_\lambda}$ denote the projection mapping. Then the representation $j:= \bigoplus_\lambda \gamma_\lambda\circ\iota$ is completely isometric. Moreover, since $\gamma_\lambda\circ\iota$ is $*$-RFD for every $\lambda$, it follows that $j$ is $*$-RFD by \cite[Lemma 2.2]{clouatre2021finite}.

{\rm (ii)}$\Rightarrow${\rm (iii)}: Obvious.

{\rm (iii)}$\Rightarrow${\rm (i)}: Let $\iota:\A\rightarrow B(\H)$ be a completely isometric RFD representation. Then there is a net of representations $\pi_\lambda:\A\rightarrow B(\H), \lambda\in\Lambda,$ such that $\rC^*(\pi_\lambda(\A))\H$ is finite-dimensional and $\pi_\lambda(a)$ converges strongly to $\iota(a)$ for each $a\in\A$. For each $\lambda$, we may find a positive integer $r_\lambda$ and a unitary operator $U_\lambda: \rC^*(\pi_\lambda(\A))\H\rightarrow\bC^{r_\lambda}$. Then $\bigoplus_\lambda U_\lambda\pi_\lambda(\cdot)U_\lambda^*$ is a completely isometric representation of $\A$.
\end{proof}

Now we identify two constructions for $\rC^*$-covers of an RFD operator algebra. This is similar to the concrete representation of the RFD-maximal $\rC^*$-cover but we require extra machinery to account for set-theoretic technicalities. This process can be seen in the classical construction of the maximal $\rC^*$-cover \cite[Proposition 2.4.2]{blecher2004operator}. Our notation will be consistent throughout.

Let $\A$ be an RFD operator algebra. Suppose that the cardinality of $\A$ is less than or equal to a cardinal $\kappa$ satisfying $\kappa^{\aleph_0} = \kappa$. For each cardinal $\alpha\leq\kappa,$ fix a Hilbert space $\H_\alpha$ of dimension $\alpha$. Let $\F$ denote the set of all RFD representations $\rho:\A\rightarrow B(\H_\alpha), \alpha\leq \kappa$. Define a representation $\upsilon_s= \bigoplus_{\rho\in\F}\rho$. Necessarily, $\kappa$ is infinite. In particular, $\F$ contains a representative from each unitary equivalence class consisting of representations of $\A$ which act on a finite-dimensional Hilbert space. As $\A$ is RFD, $\upsilon_s$ is completely isometric. Letting $\fR_s(\A) = \rC^*(\upsilon_s(\A))$, we see that $(\fR_s(\A), \upsilon_s)$ is a $\rC^*$-cover of $\A$.

Similarly, let $\G$ denote the set of all $*$-RFD representations $\rho:\A\rightarrow B(\H_\alpha)$ where $\alpha$ is a cardinal such that $\alpha\leq \kappa$. Define a representation $\upsilon_{*s} = \bigoplus_{\rho\in\G}\rho$. Similarly, the map $\upsilon_{*s}$ is completely isometric. Letting $\fR_{*s}(\A) = \rC^*(\upsilon_{*s}(\A))$, we have that $(\fR_{*s}(\A), \upsilon_{*s})$ is a $\rC^*$-cover of $\A$.

It is not immediately clear whether either of the $\rC^*$-algebras, $\fR_s(\A)$ or $\fR_{*s}(\A)$, generate RFD $\rC^*$-algebras. We will work towards showing that the latter does in fact generate an RFD $\rC^*$-algebra. Due to the succeeding Lemma \ref{L:RFDstarRFDorder}, this will imply that the $\rC^*$-cover $(\fR_{*s}(\A), \upsilon_{*s})$ is equivalent to the RFD-maximal $\rC^*$-cover.

\begin{lemma}\label{L:RFDstarRFDorder}
Let $\A$ be an RFD operator algebra. Then, we have that $(\fR(\A), \upsilon)\preceq(\fR_{*s}(\A), \upsilon_{*s})\preceq(\fR_s(\A), \upsilon_s)$.
\end{lemma}

\begin{proof}
Note that the subsets $\F$ and $\G$ defined above satisfy $\G\subset\F$. Whence, there is a natural projection mapping $\Theta: \fR_s(\A)\rightarrow\fR_{*s}(\A)$ such that $\Theta\circ\upsilon_s = \upsilon_{*s}.$ So $(\fR_{*s}(\A), \upsilon_{*s})\preceq(\fR_s(\A), \upsilon_s).$ 

To show that $(\fR(\A), \upsilon)\preceq(\fR_{*s}(\A), \upsilon_{*s})$, we use minimality of Theorem \ref{T:concreterepn}. Let $\rho:\A\rightarrow B(\H_\rho)$ be a representation of $\A$ on a finite-dimensional Hilbert space. Then, there is $n\in\bN$ and a unitary operator $U_\rho:\H_n\rightarrow\H_\rho$ such that $\chi = U_\rho^*\rho(\cdot)U_\rho:\A\rightarrow B(\H_n)$. Similar to the proof of Theorem \ref{T:concreterepn}, the construction of $\fR_{*s}(\A)$ allows us to obtain a $*$-representation $\pi:\fR_{*s}(\A)\rightarrow B(\H_n)$ such that $\pi\circ\upsilon_{*s} = \chi$. Define a $*$-representation $\theta:\fR_{*s}(\A)\rightarrow B(\H_\rho)$ by $\theta = U\pi(\cdot) U^*$. Then $\theta\circ\upsilon_{*s} = \rho$. Therefore, $(\fR(\A), \upsilon)\preceq(\fR_{*s}(\A), \upsilon_{*s})$.
\end{proof}

We now identify how the two new $\rC^*$-covers are equipped with lifting properties similar to that exposed in Theorem \ref{T:concreterepn} with the RFD-maximal $\rC^*$-cover.

\begin{proposition}\label{P:RFDandStarRFDlifts}
Let $\A$ be an RFD operator algebra and $\rho:\A\rightarrow B(\H)$ be a representation. Then, the following statements hold: \begin{enumerate}[{\rm (i)}]
\item If $\rho$ is RFD, then there exists a unique $*$-representation $\theta:\fR_s(\A)\rightarrow B(\H)$ such that $\theta\circ\upsilon_s = \rho$;
\item If $\rho$ is $*$-RFD, then there exists a unique $*$-representation $\theta:\fR_{*s}(\A)\rightarrow B(\H)$ such that $\theta\circ\upsilon_{*s} = \rho$.
\end{enumerate}
\end{proposition}

\begin{proof}
We show {\rm (ii)} with the proof of {\rm (i)} being similar. First, assume that $\dim\H\leq\kappa$. It is easy to verify that $*$-RFD representations are stable under unitary equivalence. So we may find a unitary operator $U$ such that $\chi = U^*\rho(\cdot)U\in\G$. Define $\pi:\fR_{*s}(\A)\rightarrow B(\H_\chi)$ by $\pi(T) = P_{\H_\chi}T\mid_{\H_\chi}$. By construction of $\fR_{*s}(\A)$, we see that $\pi$ is a $*$-representation.  If we let $\theta = U\pi(\cdot)U^*$, then $\theta\circ\upsilon_{*s} = \rho$ as desired.

Now, suppose that $\dim\H>\kappa$. By \cite[Proposition 2.4.2]{blecher2004operator}, we may express $\rho = \bigoplus_i\pi_i$ where each $\pi_i:\A\rightarrow B(\K_i)$ is a representation of $\A$ on a Hilbert space with dimension at most $\kappa$. Due to \cite[Lemma 2.4]{clouatre2021finite}, each $\pi_i$ is also a $*$-RFD representation of $\A$ (or RFD for statement {\rm (i)}). By the previous paragraph, for each $i$, we obtain a $*$-representation $\theta_i:\fR_{*s}(\A)\rightarrow B(\K_i)$ such that $\theta_i\circ\upsilon_{*s} = \pi_i$. Taking $\theta = \bigoplus_i\theta_i$, we see that $\theta\circ\upsilon_{*s} = \rho$ as desired. Uniqueness holds for the same reason as in Theorem \ref{T:concreterepn}.
\end{proof}

We remark that Proposition \ref{P:RFDandStarRFDlifts} makes no statement on the residually finite-dimensionality of the resulting $*$-representations. This will be clarified by showing that the lift of a $*$-RFD representation as in Proposition \ref{P:RFDandStarRFDlifts} is necessarily an RFD $*$-representation. However, we do not know of a corresponding statement about RFD representations. The following is a refinement of \cite[Theorem 3.3]{clouatre2021finite}.

\begin{theorem}\label{T:LiftStarRFD}
Let $\A$ be an operator algebra and $\rho:\A\rightarrow B(\H)$ be a representation. Let $\theta:\rC^*_{max}(\A)\rightarrow B(\H)$ denote the $*$-representation satisfying $\theta\circ\mu=\rho$. Then, $\rho$ is a $*$-RFD representation if and only if $\theta$ is an RFD $*$-representation.
\end{theorem}

\begin{proof}
$(\Rightarrow)$ Let $\rho_\lambda:\A\rightarrow B(\H), \lambda\in\Lambda$, be a net of representations such that $\rC^*(\rho_\lambda(\A))\H$ is finite-dimensional and \[*\sot\lim_\lambda \rho_\lambda(a) = \rho(a), \ \ \ \ \ \ a\in\A.\]For each $\rho_\lambda$, we obtain a corresponding $*$-representation $\theta_\lambda: \rC^*_{max}(\A)\rightarrow B(\H)$ satisfying $\theta_\lambda\circ\mu=\rho_\lambda$. Note that \[\theta_\lambda(\rC^*_{max}(\A))\H = \rC^*(\theta_\lambda\circ\mu(\A))\H = \rC^*(\rho_\lambda(\A))\H\]is finite-dimensional for each $\lambda$. For any $w\in\rC^*_{max}(\A)$ which is a finite sum of words in $\mu(\A)\cup\mu(\A)^*$, we obtain that\[*\sot\lim_\lambda\theta_\lambda(w)=\theta(w)\]as the adjoint, addition and multiplication are jointly continuous over bounded sets. This set is dense in $\rC^*_{max}(\A)$ and so $\theta$ is an RFD $*$-representation.

$(\Leftarrow)$ Suppose that $\theta_\lambda, \lambda\in\Lambda$, is a net of $*$-representations such that $\theta_\lambda(\rC^*_{max}(\A))\H$ is finite-dimensional and \[\sot\lim_\lambda\theta_\lambda(t) = \theta(t), \ \ \ \ \ \ \ t\in\rC^*_{max}(\A).\]Then the net $(\theta_\lambda\circ\mu)_\lambda$ demonstrates that $\rho$ is a $*$-RFD representation of $\A$. 
\end{proof}

As a result, we have that $(\fR_{*s}(\A), \upsilon_s)$ is equivalent to the RFD-maximal $\rC^*$-cover:

\begin{corollary}\label{C:RFDmaxLiftsStarRFD}
If $\A$ is an RFD operator algebra, then $(\fR(\A), \upsilon)\sim(\fR_{*s}(\A), \upsilon_{*s})$. In particular, for any $*$-RFD representation $\rho:\A\rightarrow B(\H_\rho)$, there is a unique $*$-representation $\theta:\fR(\A)\rightarrow B(\H_\rho)$ such that $\theta\circ\upsilon = \rho$. Furthermore, $\theta$ is an RFD $*$-representation.
\end{corollary}

\begin{proof}
As $\upsilon_{*s}:\A\rightarrow B(\H)$ is a direct sum of $*$-RFD representations, it follows that $\upsilon_{*s}$ is also a $*$-RFD representation \cite[Lemma 2.2]{clouatre2021finite}. Let $q:\rC^*_{max}(\A)\rightarrow \fR_{*s}(\A)$ be the surjective $*$-representation satisfying $q\circ\mu = \upsilon_{*s}$. By Theorem \ref{T:LiftStarRFD}, we have that $q$ is an RFD $*$-representation.

Let $\theta_\lambda:\rC^*_{max}(\A)\rightarrow B(\H), \lambda\in\Lambda,$ be a net of $*$-representations such that $\theta_\lambda(\rC^*_{max}(\A))\H$ is finite-dimensional for each $\lambda$ and \[\sot\lim_\lambda\theta_\lambda(t)=q(t), \ \ \ \ \ \ t\in\rC^*_{max}(\A).\]For each $\lambda\in\Lambda$, define a representation of $\A$ by $\rho_\lambda=\theta_\lambda\circ\mu$. By Theorem \ref{T:LiftStarRFD}, $\rho_\lambda$ is a $*$-RFD representation of $\A$. Hence, by Proposition \ref{P:RFDandStarRFDlifts} {\rm (ii)}, there is a $*$-representation $\widehat{\theta}_\lambda: \fR_{*s}(\A)\rightarrow B(\H)$ such that $\widehat{\theta}_\lambda\circ\upsilon_{*s}=\rho_\lambda.$ Observe that \[\widehat{\theta}_\lambda\circ q\circ \mu = \widehat{\theta}_\lambda\circ\upsilon_{*s} = \rho_\lambda = \theta_\lambda\circ\mu.\]Hence, $\widehat{\theta}_\lambda\circ q= \theta_\lambda$ and so\[\sot\lim_\lambda \widehat{\theta}_\lambda\circ q(t)  = q(t), \ \ \ \ \ \ t\in\rC^*_{max}(\A).\]Since $\widehat{\theta}_\lambda(\fR_{*s}(\A))\H = \theta_\lambda(\rC^*_{max}(\A))\H$ is finite-dimensional for each $\lambda\in\Lambda,$ we obtain that the identity representation of $\fR_{*s}(\A)$ is an RFD $*$-representation. By Proposition \ref{P:cisRFD} (or \cite[Theorem 2.4]{exel1992finite}), we obtain that $\fR_{*s}(\A)$ is an RFD $\rC^*$-algebra. So $(\fR_{*s}(\A), \upsilon_{*s})\preceq(\fR(\A), \upsilon)$ and Lemma \ref{L:RFDstarRFDorder} implies that the $\rC^*$-covers are equivalent. The last two statements are direct consequences of Proposition \ref{P:RFDandStarRFDlifts} {\rm (ii)} and Theorem \ref{T:LiftStarRFD}.
\end{proof}

The reader should exercise some care in utilizing Corollary \ref{C:RFDmaxLiftsStarRFD}. Indeed, let $\A$ be an RFD operator algebra and $(\fA, \iota)$ be a $\rC^*$-cover such that $(\fA, \iota)\preceq(\fR(\A), \upsilon)$. Suppose $\rho$ is a representation of $\A$ which lifts to a $*$-representation $\theta$ of $\fA$. Then, $\rho$ is a $*$-RFD representation and the lift of $\rho$ to $\fR(\A)$ is an RFD $*$-representation. Withstanding this, it is not necessarily true that $\theta$ itself will be RFD:

\begin{example}\label{E:ShiftStarRFD}
Let $\H$ be a separable, infinite-dimensional Hilbert space with orthonormal basis $\{e_n : n\in\bN\}$. For each $n\in\bN$, let $P_n$ be the orthogonal projection onto the subspace spanned by $\{e_1, \ldots, e_n\}$. Let $\A$ denote the operator algebra consisting of triangular operators. That is, $T\in\A$ precisely when $P_n TP_n = TP_n$ for every $n\in\bN$. It is known that $\A$ is an RFD operator algebra \cite[Example 3.3]{clouatre2019residual}. Also, note that $\rC^*(\A)$ is the Toeplitz algebra.

Define a representation $\pi_n:\A\rightarrow B(\H)$ by $\pi_n(T) = TP_n$. We see that $\rC^*(\pi_n(\A))\H$ is finite-dimensional and that \[*\sot\lim_{n\rightarrow\infty}\pi_n(T) = T, \ \ \ \ \ \ T\in\A,\]as $(P_n)_{n\geq1}$ converges strongly to the identity operator on $\H$. Therefore the identity representation of $\A$ is $*$-RFD. But the identity representation of $\A$ lifts to the identity representation of $\rC^*(\A)$, which is not an RFD $*$-representation. Indeed, otherwise $\rC^*(\A)$ is an RFD $\rC^*$-algebra by Proposition \ref{P:cisRFD}. This is a contradiction because $\rC^*(\A)$ contains $\fK(\H)$, which is not RFD.
\end{example}

In what follows next, we identify the connection between RFD and $*$-RFD representations. Although the SOT and $*$-SOT topologies are distinct, we are unable to determine whether RFD and $*$-RFD representations actually differ. For Question \ref{Q:main} to hold in the affirmative, it would be necessary that RFD and $*$-RFD representations coincide for RFD operator algebras. We now reflect on this and summarize when it is true. This reconceptualizes \cite[Theorem 3.3]{clouatre2021finite} by using intermediate $\rC^*$-covers.

\begin{theorem}\label{T:starRFDequivalence} 
Let $\A$ be an RFD operator algebra. Consider the following statements: \begin{enumerate}[{\rm (i)}]
\item $\rC^*_{max}(\A)$ is an RFD $\rC^*$-algebra;
\item for any completely isometric RFD representation $\iota:\A\rightarrow B(\H)$, we have that $(\rC^*(\iota(\A)), \iota)\preceq(\fR(\A), \upsilon)$;
\item $(\fR_s(\A), \upsilon_s)\sim(\fR(\A), \upsilon)$;
\item every RFD representation of $\A$ is $*$-RFD;
\item for any RFD representation $\rho:\A\rightarrow B(\H)$, there is an RFD $*$-representation $\theta:\rC^*_{max}(\A)\rightarrow B(\H)$ satisfying $\theta\circ\mu=\rho$;
\end{enumerate}Then, we have that\[{\rm (i)}\Longrightarrow{\rm (ii)}\Longleftrightarrow{\rm (iii)}\Longleftrightarrow{\rm (iv)}\Longleftrightarrow{\rm (v)}.\]If, in addition, every representation of $\A$ is RFD, then ${\rm (v)}\Longrightarrow{\rm (i)}$.
\end{theorem}

\begin{proof}
{\rm (i)}$\Rightarrow${\rm (ii)}: When $\rC^*_{max}(\A)$ is an RFD $\rC^*$-algebra, we have that $(\fR(\A), \upsilon)\sim(\rC^*_{max}(\A), \mu)$ by Corollary \ref{C:RFDmaxiffmaxRFD}. The statement is then trivial.

{\rm (ii)}$\Rightarrow${\rm (iii)}: The representation $\upsilon_s$ is RFD by \cite[Lemma 2.2]{clouatre2021finite}. So $(\fR_s(\A), \upsilon_s)\preceq(\fR(\A), \upsilon)$. By Lemma \ref{L:RFDstarRFDorder}, we infer that {\rm (iii)} holds.

{\rm (iii)}$\Rightarrow${\rm (iv)}: Let $\rho:\A\rightarrow B(\H)$ be an RFD representation. By Proposition \ref{P:RFDandStarRFDlifts}, there exists a $*$-representation $\theta:\fR_s(\A)\rightarrow B(\H)$ such that $\theta\circ\upsilon_s = \rho$. The assumption implies that $\fR_s(\A)$ is an RFD $\rC^*$-algebra. By \cite[Theorem 2.4]{exel1992finite}, $\theta$ is a $*$-RFD representation. Let $q: \rC^*_{max}(\A)\rightarrow \fR_s(\A)$ be the surjective $*$-representation satisfying $q\circ\mu = \upsilon_s$. Then $\theta\circ q$ is also a $*$-RFD representation. Hence, by Theorem \ref{T:LiftStarRFD}, we obtain that \[\theta\circ q\circ \mu = \theta\circ\upsilon_s=\rho\]is a $*$-RFD representation of $\A$.

{\rm (iv)}$\Rightarrow${\rm (v)}: This is Theorem \ref{T:LiftStarRFD}.

{\rm (v)}$\Rightarrow${\rm (ii)}: Let $\iota:\A\rightarrow B(\H)$ be a completely isometric RFD representation. By assumption, there is an RFD $*$-representation $\theta:\rC^*_{max}(\A)\rightarrow B(\H)$ satisfying $\theta\circ\mu = \iota$. By Theorem \ref{T:LiftStarRFD}, $\iota$ is in fact a $*$-RFD representation. Whence, by Corollary \ref{C:RFDmaxLiftsStarRFD}, there is a $*$-representation $\beta:\fR(\A)\rightarrow B(\H)$ such that $\beta\circ\upsilon = \iota$. We have that \[\beta(\fR(\A)) = \beta(\rC^*(\upsilon(\A))) = \rC^*(\iota(\A))\]and so $(\rC^*(\iota(\A)), \iota)\preceq(\fR(\A), \upsilon)$.

Now, assume that every representation of $\A$ is RFD and that {\rm (v)} holds. In particular, $\mu$ is an RFD representation and, by assumption, the identity representation of $\rC^*_{max}(\A)$ is RFD. Then, Proposition \ref{P:cisRFD} implies that $\rC^*_{max}(\A)$ is RFD.
\end{proof}

We conclude by remarking that, to show the validity of Question \ref{Q:main}, it suffices to prove two statements. Namely, it would be equivalent to show that {\rm (i)} every representation of $\A$ is RFD and {\rm (ii)} every RFD representation of $\A$ is $*$-RFD. If {\rm (i)} and {\rm (ii)} are true, then $\rC^*_{max}(\A)$ is RFD by \cite[Theorem 3.3]{clouatre2021finite}. We have not been able to find a counterexample to {\rm (ii)}, but despite this, there is essentially a unique possibility which could arise.

\begin{remark}\label{R:RFDbutNotStarRFD}
Let $\A$ be a unital operator algebra and suppose that $\iota:\A\rightarrow B(\H)$ is a unital completely isometric RFD representation. Let $\pi_\lambda:\A\rightarrow B(\H), \lambda\in\Lambda,$ be a net of completely contractive representations such that $\rC^*(\pi_\lambda(\A))\H$ is finite-dimensional and $(\pi_\lambda(a))_\lambda$ converges strongly to $\iota(a)$ for every $a\in\A$. For each $\lambda$, let $P_\lambda$ be the finite-rank projection onto \[\E_\lambda = \overline{\rC^*(\pi_\lambda(\A))\H}.\]We utilize the following standard fact: whenever $F\in B(\H)$ is a finite-rank operator and $(S_\lambda)_\lambda$ is a net in $B(\H)$ which converges strongly to $S$, then $(FS_\lambda^*)_\lambda$ converges to $FS^*$ in the norm topology of $B(\H)$. So \[\lim_\lambda P_\mu\pi_\lambda(a)^* = P_\mu\iota(a)^*, \ \ \ \ \ \ \mu\in\Lambda,a\in\A,\]in the norm topology of $B(\H)$.

If $(\pi_\lambda(a)^*)_\lambda$ converges strongly to an operator $T(a)\in B(\H)$ for some $a\in\A,$ then $P_\mu T(a) = P_\mu\iota(a)^*$ for each $\mu\in\Lambda$.  By \cite[Lemma 3.1]{exel1992finite}, the net $(P_\lambda)_\lambda$ converges strongly to the identity operator and so we obtain that $T(a) = \iota(a)^*$.
\end{remark}

\section{Hadwin liftings for operator algebras}\label{S:Hadwin}

To close, we obtain a non self-adjoint version of Hadwin's characterization of separable RFD $\rC^*$-algebras \cite{hadwin2014lifting}. We will recount the details of their work here.

Let $\{ e_n : n\in\bN\}$ be an orthonormal basis for $\ell^2$. For each $n\in\bN$, let $P_n$ be the orthogonal projection onto the linear span of $\{ e_1, \ldots, e_n\}$ and let $\M_n = P_n B(\ell^2)P_n$. Let $\fB$ be the $\rC^*$-subalgebra of $\prod_{n=1}^\infty\M_n$ consisting of elements $(T_n)_{n\geq1}$ for which there is an operator $T\in B(\ell^2)$ such that \[*\sot\lim_{n\rightarrow\infty}T_n= T.\]Define a $*$-representation \[\pi:\fB\rightarrow B(\ell^2), \ \ \ \ \ (T_n)_{n\geq1}\mapsto *\sot\lim_{n\rightarrow\infty}T_n.\]It is routine to check that $\pi$ is a surjective unital $*$-representation \cite[Lemma 1]{hadwin2014lifting}. Let $\A$ be an operator algebra. We say that a representation $\rho:\A\rightarrow B(\ell^2)$ is \emph{$*$-liftable in the sense of Hadwin} if there is a representation $\tau:\A\rightarrow\fB$ such that $\pi\circ\tau = \rho$. Recall that $\widetilde{\fA}$ denotes the unitization of a $\rC^*$-algebra $\fA$. Hadwin's result is then stated as follows \cite[Theorem 11]{hadwin2014lifting}:

\begin{theorem}\label{T:Hadwin}
Let $\fA$ be a separable $\rC^*$-algebra. Then, $\fA$ is an RFD $\rC^*$-algebra if and only if every unital $*$-representation $\sigma:\widetilde{\fA}\rightarrow B(\ell^2)$ is $*$-liftable in the sense of Hadwin.
\end{theorem}

We now present our non-self adjoint version to Theorem \ref{T:Hadwin}.

\begin{theorem}\label{T:nonSAHadwin}
Let $\A$ be a separable operator algebra. Then, $\rC^*_{max}(\A)$ is an RFD $\rC^*$-algebra if and only if every unital representation $\rho:\widetilde{\A}\rightarrow B(\ell^2)$ is $*$-liftable in the sense of Hadwin.
\end{theorem}

\begin{proof}
$(\Rightarrow)$ If $\rC^*_{max}(\A)$ is an RFD $\rC^*$-algebra, then $\rC^*_{max}(\widetilde{\A})$ is also RFD by \cite[2.4.3]{blecher2004operator}. Hence, we have that $(\rC^*_{max}(\widetilde{\A}), \mu_1)$ is the RFD-maximal $\rC^*$-cover of $\widetilde{\A}$.

Let $\rho:\widetilde{\A}\rightarrow B(\ell^2)$ be a unital representation. Then, there is a $*$-representation $\theta: \rC^*_{max}(\widetilde{\A})\rightarrow B(\ell^2)$ such that $\theta\circ\mu_1 = \rho$. As in Remark \ref{R:ConcreteRFDunital}, $\mu_1$ may taken to be unital. Hence, $\theta$ is also unital. As $\rC^*_{max}(\widetilde{\A})$ is a separable RFD $\rC^*$-algebra, Theorem \ref{T:Hadwin} yields that there is a unital $*$-representation $\sigma:\rC^*_{max}(\widetilde{\A})\rightarrow\fB$ such that $\pi\circ\sigma = \theta$. Then we see that \[\pi\circ\sigma\circ\mu_1 = \theta\circ\mu_1 = \rho.\]Taking $\tau = \sigma\circ\mu_1$ shows that $\rho$ is $*$-liftable in the sense of Hadwin.

$(\Leftarrow)$ By \cite[2.4.3]{blecher2004operator} and \cite[Theorem 3.3]{clouatre2021finite}, it suffices to show that every representation of $\widetilde{\A}$ is $*$-RFD. We first show that every unital representation $\rho:\widetilde{\A}\rightarrow B(\ell^2)$ is $*$-RFD. As $\rho$ is $*$-liftable in the sense of Hadwin, there is a representation $\tau:\widetilde{\A}\rightarrow\fB$ such that $\pi\circ\tau = \rho$. By definition of $\fB$, we see that for each $n\in\bN$, there is a representation $\tau_n:\widetilde{\A}\rightarrow\M_n$ and $\tau(a) = (\tau_n(a))_{n\geq 1}$. Note that $\rC^*(\tau_n(\widetilde{\A}))\ell^2\subset\M_n$ is finite-dimensional for each $n\in\bN$. Also, \[\rho(a) = \pi\circ\tau(a) = *\sot\lim_{n\rightarrow\infty} \tau_n(a), \ \ \ \ \ \ a\in\widetilde{\A},\] and so $\rho$ is a $*$-RFD representation.
 
Now let $\rho:\widetilde{\A}\rightarrow B(\H)$ be an arbitrary representation. If $\rho$ is non-unital, then we may express $\rho = \widehat{\rho}\oplus 0$ where $\widehat{\rho}$ is a unital representation of $\widetilde{\A}$. We have that $\rho$ is $*$-RFD if and only if $\widehat{\rho}$ is $*$-RFD. Whence, we may assume that $\rho$ is unital. Further, we may assume the non-trivial case of when $\H$ is infinite-dimensional. For each $h\in\H$, let $\H_h = \ol{\text{span}}\{\rC^*(\rho(\widetilde{\A}))h\}$. Note that $\H_h$ is a separable Hilbert space which is reducing for $\rho(\widetilde{\A})$. So we may express $\rho = \bigoplus_{h\in\H}\rho\mid_{\H_h}$ and, for each $h\in\H$, $\rho\mid_{\H_h}$ is a representation of $\widetilde{\A}$ on a separable Hilbert space. Hence, $\rho\mid_{\H_h}$ is unitarily equivalent to a unital representation of $\widetilde{\A}$ on $\ell^2$ and so $\rho\mid_{\H_h}$ is $*$-RFD. Therefore $\rho$ is a $*$-RFD representation by \cite[Lemma 2.2]{clouatre2021finite}.
\end{proof}

\bibliographystyle{plain}
\bibliography{RFDmax}
	
\end{document}